\documentclass[11pt,a4paper]{amsart}
\usepackage{amssymb}
\usepackage{graphicx}
\usepackage{dcpic,pictex}
\usepackage[noadjust]{cite}
\usepackage{hyperref}

\theoremstyle{plain}
\newtheorem{theorem}{Theorem}[section]
\newtheorem{corollary}[theorem]{Corollary}
\newtheorem{lemma}[theorem]{Lemma}
\newtheorem{proposition}[theorem]{Proposition}

\theoremstyle{definition}

\theoremstyle{remark}
\newtheorem{remark}{Remark}[section]

\numberwithin{equation}{section}

\numberwithin{table}{section}

\numberwithin{figure}{section}

\begin{document}
	
\title[Gelfand-Tsetlin modules for Lie algebras of rank $2$]{Gelfand-Tsetlin modules for Lie algebras of rank $2$}

	
\author{Milica An\dj eli\'c}
\address{Department of Mathematics, Kuwait University, Al-Shadadiyah, Kuwait}
\email{milica.andelic@ku.edu.kw}
	
\author{Carlos M. da Fonseca}
\address{Kuwait College of Science and Technology, Doha District, Safat
	13133, Kuwait}
\email{c.dafonseca@kcst.edu.kw}
\address{Faculty of Applied Mathematics and Informatics, Technical
	University of Sofia, Kliment Ohridski Blvd. 8, 1000 Sofia, Bulgaria}
\email{carlos.fonseca@tu-sofia.bg}
\address{Chair of Computational Mathematics, University of Deusto, 48007
	Bilbao, Spain}

\author{Vyacheslav Futorny}
\address{Shenzhen International Center for Mathematics, Southern University of Science and Technology, China}
\email{vfutorny@gmail.com}

\author{Andrew Tsylke}
\address{Kyiv Taras Shevchenko University, Kyiv, Ukraine}
\email{andrew4tsylke@gmail.com}
	
	
	\date{\today}
	
	
\begin{abstract}
		We explicitly construct families of simple modules for all simple Lie algebras of rank $2$ 
		on which a certain commutative subalgebra acts diagonally with a simple spectrum. 
		In type $A$, these modules are the well-known generic Gelfand-Tsetlin modules.
\end{abstract}
	
	\maketitle



\section{Introduction.}
\addcontentsline{toc}{section}{Section1.}

Let 
$\mathfrak{g}$ 
be a simple finite dimensional simple Lie algebra over the complex numbers and let
$\mathfrak{h}$ 
be a fixed Cartan subalgebra of 
$\mathfrak{g}$.
A $\mathfrak{g}$-module $M$ is \emph{weight} (with respect to $\mathfrak{h}$) if $\mathfrak{h}$ is diagonalizable on $M$, that is
$$M=\bigoplus_{\lambda\in \mathfrak{h}^*}M_{\lambda},$$ where $hv=\lambda(h)v$ for any $v\in M_{\lambda}$ and $h\in \mathfrak h$. The subspace $M_{\lambda}$ is called a weight subspace of weight $\lambda$, if $M_{\lambda}\neq 0$.

Simple weight modules were studied extensively in the last 50 years. Classical results of Fernando \cite{F1990} and Mathieu \cite{M2000} provided a complete classification of simple weight modules with finite dimensional weight subspaces. On the other hand, the classification of simple weight modules with infinite dimensional weight subspaces is still an open problem.
The most progress has been achieved in the case of Lie algebras of type A, where
simple Gelfand-Tsetlin modules were classified (see \cite{FGR2019, RZ2018, V2018, W2024} and references therein). These are weight modules with diagonalizable action of a certain commutative subalgebra of the universal enveloping 
algebra $U(\mathfrak{g})$, called the Gelfand-Tsetlin subalgebra. Generically, such simple Gelfand-Tsetlin modules have  infinite dimensional weight subspaces. In particular, in the case of $\mathfrak{sl}(2)$ we obtain in this way all simple weight modules. They depend on two parameters and have $1$-dimensional weight subspaces. 
In the case of $\mathfrak{g}=\mathfrak{sl}(3)$ a complete description of simple Gelfand-Tsetlin modules was given in \cite{FGR2021}. 

The original approach to the study of weight modules was based on the reduction to the study of simple modules over the centralizer $U_0(\mathfrak{g})$ of the Cartan subalgebra $\mathfrak{h}$ in
 the universal enveloping algebra $U(\mathfrak{g})$: if $M$ is a simple weight $\mathfrak{g}$-module then 
 $M_\lambda$ is a simple $U_0(\mathfrak{g})$-module for any weight $\lambda$ of $M$.
 Hence, 
  every simple weight $\mathfrak{g}$-module corresponds to a  simple (not unique) $U_0(\mathfrak{g})$-module and, in its turn, any simple $U_0(\mathfrak{g})$-module corresponds to a unique simple weight $\mathfrak{g}$-module. This approach was successfully used in the case of $\mathfrak{g}=\mathfrak{sl}(3)$ \cite{BFL1995, BL1982a, BL1982b, BL1987,  DFO1989, DFO1982, F1990, F1991}, etc.
  
  As the structure of  $U_0(\mathfrak{g})$ is rather complicated (there are two commuting generators for $\mathfrak{sl}(2)$, and there are six generators  for $\mathfrak{sl}(3)$ with three polynomial relations and five commuting generators among them), there were essentially no attempts beyond the $\mathfrak{sl}(3)$ case.

The aim of this paper is to revise the centralizer approach and to construct new simple weight modules with infinite dimensional weight subspaces for all simple Lie algebras of rank $2$. 
 We explicitly construct  simple generic modules in the category  of $\Gamma$-pointed modules for a  commutative subalgebra $\Gamma$ of  the centralizer $U_0(\mathfrak{g})$.
In type $A$, $\Gamma$-pointed modules are the celebrated Gelfand-Tsetlin modules, and similar constructions can be viewed analogously in other types.


The structure of the paper is the following. 
In Section 2 we discuss the structure of  the centralizer $U_0(\mathfrak{g})$ of the Cartan subalgebra $\mathfrak{h}$ in
 the universal enveloping algebra $U(\mathfrak{g})$, prove that 
$U_0(\mathfrak{g})$ is finitely generated and finitely presented.
We  give a  generating set of elements and describe an algorithm for computing all relations between them. In Section 4 we consider the Lie algebra of type
$A_2$. Our approach is a suitable modification of \cite{BL1982a} and \cite{F1989}, and we recover a construction of generic torsion free 
$A_2$-modules with infinite dimensional weight spaces obtained in \cite{F1989} and \cite{FGR2021}.  They are \emph{tame} Gelfand-Tsetlin modules with diagonalizable action of the Gelfand-Tsetlin subalgebra.
In Section 5 we consider the Lie algebra $\mathfrak{g}$ of type
$C_2$ and
give the generators and the defining relations of the centralizer $U_0(\mathfrak{g})$.
We construct two $4$-parameter families of simple  torsion free 
$C_2$-modules with infinite dimensional weight spaces. These modules are $\Gamma$-pointed, where $\Gamma$ is the $4$-generated Gelfand-Tsetlin subalgebra of $U_0(\mathfrak{g})$ which has a simple spectrum on such representations.
Finally, in Section 6  we construct a $3$-parameter family of simple  torsion free 
$G_2$-modules with infinite dimensional weight spaces. These modules are $\Gamma$-pointed with respect to a $4$-generated Gelfand-Tsetlin subalgebra $\Gamma$ of $U_0(\mathfrak{g})$ which has a simple spectrum on such representations.

We hope in a next work to use the defined representations to construct new simple modules for all simple finite dimensional and affine Lie algebras via the parabolic induction. 

\section{Cartan centralizers}
\addcontentsline{toc}{section}{Section2.}

Let $\Delta= \{\alpha_1, \ldots, \alpha_{k_1} \}$ be the root system of 
$(\mathfrak{g},\mathfrak{h})$ and let $\pi = \{ \beta_1, \ldots, \beta_{k_0} \}$ be a basis of $\Delta$. 
With respect to the basis $\pi$, we have the
decomposition of $\Delta$ into positive and negative roots: 
$$\Delta= \Delta^+ \cup \Delta^-\, .$$  Let 
$W$ be the  Weyl group of the root system $\Delta$. 
Choose  a basis $G = G_0 \cup G_1$ of the Lie algebra  $\mathfrak{g}$, 
where  
$G_0 = \{h_\beta\in \mathfrak{h}\, | \,\beta \in \pi \}$  and
$G_1  = \{ e_\alpha\in \mathfrak{g}_{\alpha}\setminus \{0\}\, |\, \alpha \in \Delta \}$.  
Set 
$f_\alpha=e_{-\alpha}$. 

Denote by 
$U_0(\mathfrak{g})$  the centralizer of the Cartan subalgebra 
$\mathfrak{h}$  in the universal enveloping algebra 
$U(g)$. 
For each $i \in \mathbb N$, let us denote by  
$U^{(i)}(\mathfrak{g})$ the
vector subspace of
$U(\mathfrak{g})$ 
spanned by the monomials 
$x_1 x_2 \cdots x_j$, 
where 
$x_1, \ldots, x_j \in G$ 
and 
$j \le i$. Thus we get an
increasing sequence of subspaces 
\begin{equation*}
  U^{(1)}(\mathfrak{g}) \subset U^{(2)}(\mathfrak{g}) \subset \cdots \subset U^{(i)}(\mathfrak{g})  \subset \cdots, 
\label{eq:refer001} 
\end{equation*}
which defines a canonical filtration of 
$U(\mathfrak{g})$.
The canonical filtration of 
$U_0(\mathfrak{g})$ is the sequence of subspaces
\begin{equation*}
  U_0^{(1)}(\mathfrak{g}) \subset U_0^{(2)}(\mathfrak{g}) \subset \cdots \subset U_0^{(i)}(\mathfrak{g})  \subset \cdots,
\label{eq:refer002}  
\end{equation*}
where 
$U_0^{(i)}(\mathfrak{g}) = U^{(i)}(\mathfrak{g}) \cap U_0(\mathfrak{g})$.

For any monomial $X=x_1x_2\cdots x_j$ denote by $T_0(X)$ set of monomials received from $X$ by permutations of variables $x_i$.
Define the degree function 
$\deg(y)$   as follows:  
$\deg(y)=i$ if
$y    \in U^{(i)}(\mathfrak{g})$, but 
$y \notin U^{(i-1)}(\mathfrak{g})$.
    
Now, fix some order on the set 
$G$:
\begin{equation}
x_1 \leq x_2 \leq \cdots \leq x_k,
\label{eq:refer004}
\end{equation}
where $k$ is the dimension of $\mathfrak{g}$.

Define standard monomials of 
$U(\mathfrak{g})$ with respect to this order as follows:

\begin{equation*}
X_S = x_1^{s_1} x_2^{s_2} \cdots x_k^{s_k}.
\label{eq:refer0041}
\end{equation*}
where 
$S = (s_1, \ldots, s_k)$ is a $k$-tuple of nonnegative integers, with at least one $s_i\ne 0$. 

For every monomial $X$ define a lexicographic order on the set $T_0(X)$ with 
respect to the order fixed in \eqref{eq:refer004}:
if $X_1=x_{i_1}x_{i_2} \cdots x_{i_n}$ and  $X_2=x_{j_1}x_{j_2} \cdots x_{j_m}$,
$X_1,X_2 \in T_0(X)$,
then $X_1 < X_2$ if there exists an index $s \leq \min(n,m)$, such that
 $x_{i_t} = x_{j_t}$, for $t < s$ and $x_{i_s} < x_{j_s}$. 

Denote by $P$ the set of all standard monomials and by $P^{(i)}$ the set of all standard monomials of degree $i$.
Set $P_0 = P \cap U_0(\mathfrak{g})$ and $P^{(i)}_0 = P^{(i)} \cap U_0(\mathfrak{g})$.

The following lemma is an immediate consequence of the PBW theorem.

\begin{lemma}\label{lem:lemCart_21}
	\
\begin{itemize}
\item[1.] The set of the standard monomials $\Bar{P}^{(i)} = P^{(1)} \cup P^{(2)} \cup \cdots \cup P^{(i)}$
forms a basis for the vector space 
$U^{(i)}(\mathfrak{g})$.
\item[2.] The set of the  standard monomials ${\bar{P}_0^{(i)}} = P_0^{(1)} \cup P_0^{(2)} \cup \cdots \cup P_0^{(i)}$
forms a basis for the vector space 
$U_0^{(i)}(\mathfrak{g})$.
\item[3.] If 
$a \in U^{(i)}(\mathfrak{g})$ and
$b \in U^{(j)}(\mathfrak{g})$,
then
$a b \in U^{(i+j)}(\mathfrak{g})$ and 
$ab - ba \in U^{(i+j-1)}(\mathfrak{g})$.
\item[4.] If $X \in P_0^{(i)}(\mathfrak{g})$ and $X_1, X_2 \in T_0(X)$, then $X - Y \in U_0^{(i-1)}(\mathfrak{g})$. 
\item[5.] 
For every monomial $X \in U_0(\mathfrak{g})$, the minimal element of the set $T_0(X)$
with respect to the above lexicographic order is a standard monomial.
\end{itemize}
\end{lemma}

       
Denote by 
$\hat P^{(i)}_0 \subset P^{(i)}_0$
the subset of monomials of degree $i$ generated by 
$G_1$,
and set
$\hat P_0 =\cup_i \hat P^{(i)}_0$.
 Every monomial $X \in U(\mathfrak{g})$ is product of elements $e_\alpha, \alpha \in \Delta$ and $h_\beta, \beta \in \pi$.
Let $X=h_{\beta_1}\cdots h_{\beta_m}e_{\alpha_1}\cdots e_{\alpha_n}$. 
Define two lists of roots associated with $X$  as follows: $L_0=(\beta_1 , \ldots, \beta_m)$ and
$L_1(X)=(\alpha_1 , \ldots, \alpha_m)$. 
Extend this definition to the set of all monomials: if $X' \in T_0(X)$, then $L_0(X')=L_0(X)$, $L_1(X')=L_1(X)$.

For any monomial $X$, the sum of all roots in the list $L_1(X)$ is called 
the weight of the list $L_1(X)$.
Clearly, the weight of the list $L_1(X)$ is zero if and only if $X\in U_0(\mathfrak{g})$.
Furthermore, if $L_0(X)=L_0(Y)$ and $L_1(X)=L_1(Y)$, for some monomials $X,Y$,
then $T_0(X)=T_0(Y)$.

A list $L_1(X)$ for $X \in \bar P_0$ is called decomposable if there exist  $X_1,X_2 \in \bar P_0$
such that $L_1(X) = L_1(X_1) \sqcup L_1(X_2)$ (disjoint union of two lists).
In this case, it follows that $T_0(X) = T_0(X_1 X_2)$.
Conversely, if no such decomposition exists, the list is called indecomposable.
Denote by $B(\Delta)$ the set of all zero-weight lists of roots.

Note that for certain monomials, decomposition is not unique.
For example, in the algebra $A_2$, as we will find in Section \ref{sec:A2torsionfree},  
$$(\alpha_1, \ldots, \alpha_6) 
=(\alpha_1,\alpha_6) \sqcup  (\alpha_2,\alpha_5) \sqcup  (\alpha_3,\alpha_4) 
= (\alpha_1,\alpha_2,\alpha_4) \sqcup  (\alpha_3,\alpha_5,\alpha_6)\, .$$

Denote by
$B_1(\Delta) \subset B(\Delta)$ the set of all zero-weight indecomposable lists.
Define the action of a Weyl group element $w \in W$ on the list of roots
$a = (\alpha_1, \alpha_2, \ldots , \alpha_n)$ as follows:
\begin{equation*}
w(a) = ({w \alpha_1}, {w \alpha_2}, \ldots , {w \alpha_n})\, .
\label{eq:refer0051}
\end{equation*}

Clearly, if $a$ is indecomposable, then $w(a)$ is also indecomposable, 
Moreover if $a$ has zero weight then $w(a)$ also has a zero weight.

The simplest example of an indecomposable list is a list containing only one positive or only one negative root.
Define the set of \emph{primitive} lists 
    $B_2(\Delta) \subseteq B_1(\Delta)$ as follows: 
    $r \in \hat B_2(\Delta)$ if there exists 
        $w \in W$ such that the list $w(r)$ contains only one negative or only one positive root.

Let $M_k= \{(n_1, n_2, \ldots , n_k)\, |\, n_i \in \mathbb N\}$ be the set of vectors with nonnegative integer coordinates.
Define a partial order on the set $M_k$ as follows:  
\begin{equation*}
(n_1, n_2, \ldots , n_k) \le (m_1, m_2, \ldots , m_k) \, ,
\quad \mbox{ if } 
n_i \le m_i 
\, ,\quad \mbox{for all $i=1, \ldots, k$.} 
\label{eq:refer005a}
\end{equation*}

We will use the following result \cite[Lemma 2.6.2]{D1996}.

\begin{lemma}\label{lem:lemCart_22}
If $S_{k} \subset M_k$ is any infinite subset, then there exist two elements $r_1, r_2 \in S_{k}$
    such that  $r_1 \le r_2$.
\end{lemma}

We have the following properties of indecomposable lists of roots.

\medskip
\begin{lemma}\label{lem:lemCart_23} Let $\mathfrak{g}$ be a simple finite dimensional Lie algebra with root system $\Delta$. Then 
\begin{itemize}
\item[1.] The set of indecomposable lists $B(\Delta)$ is finite.
\item[2.]  If $\mathfrak{g}\in \{A_2,  C_2, G_2\}$, then all indecomposable lists 
 are primitive and hence $B_1(\Delta) = B_2 (\Delta)$.
\end{itemize}
\end{lemma}

\begin{proof}
Define the function 
$\sigma$ on the set  $B(\Delta)$ as follows:
 $$\sigma (\alpha_1, \alpha_2, \ldots , \alpha_m) = (n_1, n_2, \ldots , n_{|\Delta|})\, ,$$
where $n_i$ is the number of occurrences of the $i$-th root of $\Delta$ (with respect to the ordering in \eqref{eq:refer004}) in the list $({\alpha_1}, {\alpha_2}, \ldots , {\alpha_m})$.
It is clear that a list $r_1$ is a sublist of a list $r_2$ if and only if 
$\sigma(r_1) \le \sigma(r_2)$.
Now, the proof of the first statement follows from the previous lemma.

The second statement for the case $A_2$ is obvious. We will now prove the second statement for the case $C_2$.
Let 
$\pi=\{\beta_{1},\beta_{2}\}$ be a basis of the root system. 
For convenience, we will represent the roots of $\Delta$ as vectors in this basis: $(i,j) = i\beta_{1} + j\beta_{2}$.
Then $\Delta = \{ (1,0),(0,1),(1,1),(2,1),(-1,0),(0,-1),(-1,-1),(-2,-1)\}$.
Let $v=(i,j) \in \Delta$ be any root. 
The action of the Weyl group is given by $w_1 (v )= (2j-i,j)$ and $w_2 (v) = (i,j-i)$, where $w_1$ and $w_2$ are simple reflections.
Let $r$ be an indecomposable (zero-weight) list of roots.
Note that it may contain multiple copies of the same root.
Then $r$ has the following properties:
 $|r| \ge 4$;  if $(i,j) \in r$ then $(-i,-j) \notin r$;
 if the sum of some three roots equals zero, then at least one of them is not 
in the list $r$. From these properties it follows that $r$ has only three different roots, and $r$ is primitive. For example, if $(2,1), (0,1) \in r$, then the third root can only be $(-1,-1)$. Since the sum of all roots in the set $r$ is zero, $r$ must contain at least 
two copies of the root $(-1,-1)$.  Since $r$ is indecomposable, we conclude that $r= \{(2,1), (0,1), (-1,-1), (-1,-1)\}$. Then $w_2  r  = \{(2,1), (0,-1), (-1,0), (-1,0)\}$ has only one positive root and $r$ is primitive.

The proof the second statement for the case $G_2$ is similar to the case $C_2$.
For sets of all primitive lists in all three cases a discussion is provided in Sections \ref{sec4}-\ref{sec6}.   
\end{proof}   


Define the following order on the set of roots $\Delta$ derived from  \eqref{eq:refer004}: for $\alpha_1,\alpha_2 \in \Delta$
$$
    \alpha_1 \leq \alpha_2 \quad \text{ if and only if } \quad e_{\alpha_1} \leq e_{\alpha_2}. 
\label{eq:refer005}
$$

A monomial $X$ is called \emph{perfect} if either
$X = h_\beta \in  G_0$,
or
$X = e_{\alpha_1}\cdot  e_{\alpha_2}\cdot \cdots \cdot e_{\alpha_n}$ and the associated list
$L_1(X)=({\alpha_1}, {\alpha_2}, \ldots, {\alpha_n})$  is both indecomposable and 
ordered according to \eqref{eq:refer005}, i.e., $\alpha_1 \leq \alpha_2 \leq \cdots \leq \alpha_n$.

Since the set of all indecomposable lists is finite, the set of all perfect monomials is also finite.
Let 
$I(\mathfrak{g})= \{ p_1, \ldots, p_q \}$ 
denote the set of all perfect monomials with degree greater than one.
Then $I(\mathfrak{g}) \cup G_0$ is a generating set of $U_0(\mathfrak{g})$ consisting of all perfect monomials.
Denote by $d_0$ the maximal degree of perfect monomials in $I(\mathfrak{g})$.

Let $J(\mathfrak{g}) = \{c_1,\ldots, c_q\}$ be the set of new indeterminates  and  $R=\{r_1 = c_1-p_1,\ldots,r_q = c_q-p_q\}$. Define
$$\hat U_0(\mathfrak{g}) = U_0(\mathfrak{g})[c_1,\ldots,c_q] / R.$$
Obviously, $\hat U_0(\mathfrak{g}) \cong U_0(\mathfrak{g})$. Hence, $\hat U_0(\mathfrak{g})$ is generated by 
 $J(\mathfrak{g}) \cup G_0$ with relations $K$ inherited from algebra $U_0(\mathfrak{g})$.




Denote by $Q_0(\mathfrak{g})$ the set of all monomials generated by $J(\mathfrak{g}) \cup G_0$ and 
define the function from $Q_0(\mathfrak{g})$ to $U_0(\mathfrak{g})$ as follows:
$$\omega: Q_0(\mathfrak{g}) \rightarrow  U_0(\mathfrak{g})\, , \quad   \omega(c_i) = p_i\, , \quad  \omega(h_i) = h_i, h_i \in G_0\, . $$

Let us fix some order on the set $J(\mathfrak{g}) \cup G_0$:
\begin{equation}
    x_1 \leq x_2 \leq \cdots \leq x_{q+k_0}\, , \quad  x_i \in J(\mathfrak{g})\cup G_0\, , \quad \mbox{for $i = 1,\ldots, q+k_0$.} 
\label{eq:refer006}
\end{equation}

For any monomial $Y \in Q_0(\mathfrak{g})$ denote by $T_1(Y)$, the set of all monomials $X\in Q_0(\mathfrak{g})$  such that $T_0(\omega(Y)) = T_0(\omega(X))$. Define a lexicographic order on the set $T_1(Y)$ with respect to the order \eqref{eq:refer006} as follows:
if $Y_1=x_{i_1}x_{i_2} \cdots x_{i_n}$,  $Y_2=x_{j_1}x_{j_2} \cdots x_{j_m}$ and $Y_1,Y_2 \in T_1(Y)$, then $Y_1 < Y_2$ if there exists an index $s \leq \min(n,m)$, such that
 $x_{i_t} = x_{j_t}$, for $t < s$ and $x_{i_s} < x_{j_s}$. 
A monomial $Y \in Q_0(\mathfrak{g})$ is called \emph{semi-perfect} if it is the minimal element of the set $T_1(Y)$ under this lexicographic order.  
Denote by $S(\mathfrak{g})$ the set of all semi-perfect monomials, and set
$S^{(i)}(\mathfrak{g}) = S(\mathfrak{g}) \cap U_0^{(i)}(\mathfrak{g})$. 

\begin{lemma}\label{lem:lemCart_25}
	\
\begin{itemize}
\item[1.] 
For any $X \in P_0$ there exists a unique semi-perfect monomial $Y \in Q_0(\mathfrak{g})$
such that $T_0(X) = T_1(Y)$.
\item[2.] 
Let $X=x_{1}x_{2}\cdots x_m$ with $x_i \in Q_0(\mathfrak{g})$,  be a  semi-perfect monomial. Then 
for any $s\in \{1, \ldots, m\}$ the monomial
$X'=x_{1}\cdots x_{s-1} x_{s+1}\cdots x_m$ is also semi-perfect.
\item[3.] There is a bijection $\psi: P_0 \rightarrow  S(\mathfrak{g})$: for every $X \in P_0$ 
there exists a unique $Y \in S(\mathfrak{g})$ such that $T_0(X) = T_1(\psi^{-1}(Y))$.
\item[4.] The set of semi-perfect monomials $S^{(i)}(\mathfrak{g})$
is a basis of the vector space $U_0^{(i)}(\mathfrak{g})$.
\end{itemize}
\end{lemma}

\begin{proof} The one-to-one  correspondence between standard monomials in $U_0^{(i)}(\mathfrak{g})$ and semi-perfect monomials in $\hat U_0^{(i)}(\mathfrak{g})$ follows from Lemma \ref{lem:lemCart_21} and the definition of semi-perfect monomials.
Taking into account that every standard monomial in $U_0^{(i)}(\mathfrak{g})$ can be expressed as a linear combination of 
semi-perfect monomials, the lemma follows.
\end{proof} 

Since the set of semi-perfect monomials $S(g)$ is a basis of the algebra $U_0(\mathfrak{g})$, then any element $x \in U_0(\mathfrak{g})$ can be written as a linear combination of semi-perfect elements. We need to determine how to multiply any two semi-perfect monomials and  express the product as a linear combination of elements of  $S(\mathfrak{g})$.

If an element $X \in U_0(\mathfrak{g})$  is written as
a linear combination of semi-perfect monomials, then
we say that $X$ is in a \emph{normal form}, denoted by $\mathcal{N}( X )$, and the process of conversion of a given element to a normal form we call \emph{normalization}.
Define now the set of relations 
\begin{equation*}
    K' = \{ c_{i_1}c_{i_2} \cdots c_{i_m} - \mathcal{N}( c_{i_1}c_{i_2} \cdots c_{i_m}) \,|\, m \le d_0, c_i \in J(\mathfrak{g}) \}
\label{eq:refer0251}
\end{equation*}
and the length function  on $K'$ as follows: 
${\rm Len}(c_{i_1}\cdots c_{i_m} - \mathcal{N}( c_{i_1} \cdots c_{i_m})) = m$.
Hence, $K'$ is the set of all relations with length less or equal to $d_0$. We are ready to state the main result about the Cartan centralizers of the universal enveloping algebras. 

\begin{theorem}\label{thm-main_26}
The ideal of relations $K$ is generated by the set of relations $K'$.  
\end{theorem}

\begin{proof}
First, we prove that the product of perfect elements can be expressed as 
a linear combination of semi-perfect elements using only the relations $K'$.
The proof follows from the reduction algorithm described below.

Let $X=x_{1}x_{2}\cdots x_{m}\in U_0^{(i)}(\mathfrak{g})$, where 
$x_{1},x_{2},\ldots, x_{m} \in J(\mathfrak{g})\cup G_0$ are perfect monomials.
We will show that using the relations $K'$, $X$ can be transformed into the normal form. 
The proof will proceed by induction on the degree of the monomials.
For $i=1$, the statement is obvious.
Assume that the statement holds for $X\in U_0^{(i-1)}(\mathfrak{g})$, that is any such element can be transformed into the normal form.
 By  Lemma \ref{lem:lemCart_25}, there exists a semi-perfect element 
$Y=y_{1}y_{2}\ldots y_{n} \in Q_0^{(i)}$, $y_{s} \in J(\mathfrak{g})\cup G_0$,
such that
$T_1(X) = T_1(Y)$.
There are two cases to consider. 
 
\begin{itemize}
\item[1.] Let $y_{1} \in G_0$. Since $L_0(X)=L_0(Y)$
by Lemma \ref{lem:lemCart_25}, then  
$y_{1}=x_{s}$,  for some $1 \le s \le m$.
As $y_{1}$ belongs to the center $U_0(\mathfrak{g})$, 
we can rewrite $X$ as 
$$X=y_{1}x_{1}x_{2}\cdots x_{s-1}x_{s+1}\cdots x_{m}\, .$$
Thus, 
$X= y_{1} \cdot X'$ and
$Y= y_{1} \cdot Y'$ where 
$T_1(X') = T_1(Y')$ and 
$X',Y' \in U_0^{(i-1)}$. 
Additionally, $Y'$ is a semi-perfect monomial. Now the statement follows 
by the induction hypothesis.
Note that we did not use any relations from the set $K'$.

\item[2.] Let $y_{1} \in J_0(\mathfrak{g})$.
It follows from the equality $L_1(X) = L_1( Y )$ that $L_1(y_{1})$ is a sublist of $L_1( X )$. 
Let $\{ z_1, z_2, \ldots, z_t\} \subset \{x_{1},x_{2},\ldots, x_{m}\}$ be a minimal set of perfect monomials such that 
$L_1(y_{1}) \sqsubset L_1(z_1 z_2 \cdots z_t )$.
Set $Z=z_1 z_2 \cdots z_t$. 
Since the cardinality of the list $L_1(y_{1})$ is bounded by 
$d_0$, we have 
$t \le d_0$.
Using the commutation relations for two perfect elements we can transform $X$ into the following form:
$$X=Z X_1 + X_2,$$ where 
$X_2 \in  U_0^{(i-1)}(\mathfrak{g})$ is a linear combination of monomials and $T_1(X)=T_1(Z X_1)$.
Let $Z_1$ be the minimal (semi-perfect) element in the set $T_1(Z)$. The first perfect element in $Z_1$ is $y_1$ and $Z_1 = y_1 Z_1'$. 
Now, applying the normalization process to the element  $Z - Z_1$ we get
 $Z= Z_1 + \mathcal{N}(Z - Z_1)$. Substituting it back into $X$, we get 
$$X=Z X_1 + X_2 = (y_1 Z_1' + \mathcal{N}(Z - Z_1)) X_1 + X_2 = y_1 X_1' + X_2',$$
where $T_1(X) = T_1(y_1 X_1')$ and $X_2'$ is a polynomial in $U_0^{(i-1)}(\mathfrak{g})$.
 We have 
$X= y_{1} \cdot X_1' + X_2'$, 
$Y= y_{1} \cdot Y'$ and 
$T_1(X') = T_1(Y')$, where 
$X',Y',X_2' \in U_0^{(i-1)}(\mathfrak{g})$ and
$Y'$ is semi-perfect monomial. The statement follows by
 the induction hypothesis.
\end{itemize}
\end{proof}

Let $A$ be an associative algebra defined by a set of generators $S$ and a set of relation $R$ (as polynomials in the generators of $S$).

The set of generators $S$ can be partitioned into three subsets, $S=S_1 \cup S_2 \cup S_3$, where
 $S_1$ generates the center of the algebra $A$ and $S_1 \cup S_2$ generates a commutative subalgebra of $A$.

The width of a monomial $X$, denoted by ${\rm width }(X)$, is defined as the number of occurrences of variables from the set $S_3$ in the monomial. The width of a polynomial $P$ is then defined as the maximal width among all monomials that make up the polynomial.

We can decompose the set of relation $R$ into a union of subsets as follows:
$R = R_0 \cup R_1 \cup \cdots $, where $R_i = \{ Y \in R\, |\, {\rm width }\, (Y) = i\}$, for $i=0,1,2, \ldots$.

Note that this decomposition is not unique and may vary depending on the choice of  generators $S$ and the subset $S_2$.
Our goal will be to identify the ``best" set of generators $S$ and ``best" decomposition, 
such that the cardinality of the set $S_3$ is minimal.    


\section{Category of $\Gamma$-pointed modules}

Let 
$\Gamma$ be a commutative subalgebra of 
$U_0(\mathfrak{g})$ such that
$\mathfrak{h} \subseteq \Gamma \subset U_0(\mathfrak{g})$. 
By ${\rm Hom}(\Gamma,\mathbb{C})$ we denote the set of all characters of $\Gamma$, i.e., the set of all $\mathbb{C}$-algebra homomorphisms from $\Gamma$ to $\mathbb{C}$.

Let $M$ be a $\Gamma$-module. For each $\chi \in {\rm Hom}(\Gamma,\mathbb{C})$  we set
\begin{equation*}
  M_\chi = \{v \in M\, |\,  av =\chi(a)v\, , \mbox{ for all } a \in \Gamma\},
\label{eq:refer008}
\end{equation*}
and call it the $\Gamma$-weight space of $M$ with weight $\chi$. When $M_\chi \neq\{0\}$, we say that $\chi$ is a $\Gamma$-weight of $M$ and the elements of $M_\chi$ are called $\Gamma$-weight vectors of weight $\chi$. If a $\Gamma$-module $M$ satisfies
\begin{equation*}
  M = \bigoplus_{\chi \in {\rm Hom}(\Gamma,\mathbb{C})} M_\chi,
\label{eq:refer009}
\end{equation*}
then we call $M$ a $\Gamma$-weight module. 
The dimension of the vector space 
$M_\chi\neq 0$ will be called the \emph{$\Gamma$-multiplicity} of $\chi$ in $M$. A module is called \emph{$\Gamma$-pointed} if the {$\Gamma$-multiplicity} of any character $\chi$  equals $1$, that is, $\Gamma$ separates the basis elements of $M$. In particular, $M$ is a \emph{tame} module with diagonalizable action of $\Gamma$. 
A weight module  $M$ is \emph{torsion free} provided  all root vectors of $\mathfrak{g}$ act injectively on  $M$. In particular, if $\Gamma =U(\mathfrak{h})$ then $\Gamma$-weight module is a classical weight module.

Suppose $\mathfrak{g}$ is of type $A$ and $\Gamma$ is a Gelfand-Tsetlin subalgebra of $\mathfrak{g}$. Then  $ \Gamma \subset U_0(\mathfrak{g})$ and every generic Gelfand-Tsetlin $\mathfrak{g}$-module is $\Gamma$-pointed. We refer to \cite{DFO1984} for details. Clearly, every finite dimensional $\mathfrak{g}$-module is also $\Gamma$-pointed. A family of simple $\Gamma$-pointed modules in type $A$ was studied in \cite{FRZ2019}. On the other hand, there exist Gelfand-Tsetlin modules which are not pointed.

\section{Construction of simple weight $A_2$-modules} \label{sec4}
In this section we consider the Lie algebra $\mathfrak{g}=\mathfrak{sl}(3)$. Even though this case is well understood we give some details to illustrate our approach.

\subsection{Centralizer of the Cartan subalgebra of $A_2$}
\label{sec:A2torsionfree}

Let $\Delta  =  \{ \alpha_1,\alpha_2,\alpha_3=\alpha_1+\alpha_2,  \alpha_4=-\alpha_1-\alpha_2, \alpha_5=-\alpha_2, \alpha_6=-\alpha_1\}    $
be the root system of $\mathfrak{g}$.
Fix a Chevalley basis $G$: $e_{10} =  E_{12},       f_{10} =  E_{21}, $
 $e_{01} =  E_{23}$,            $ f_{01} =  E_{32}$,              
$e_{11} =  E_{13}$,             $f_{11} = E_{31}$,  
$h_{10} =  E_{11} - E_{22}$,   $h_{01} =  E_{22} - E_{33}$.  

Let us define the following order on the elements of $G$:
\begin{equation*}
h_{01} < h_{10} < f_{01} < f_{10} < f_{11} < e_{11} < e_{10} < e_{01}.  
\label{eq:REF100}
\end{equation*}

We have the following set of indecomposable lists of roots:
$
\{\alpha_1, \alpha_6\}$, 
$\{\alpha_2, \alpha_5\}$,
$ \{\alpha_3, \alpha_4\}$,
$ \{\alpha_1, \alpha_2, \alpha_4\}$, 
$ \{\alpha_3, \alpha_5, \alpha_6\}$,                       
and the following set of perfect monomials:
 $$
 h_1 =  h_{01}, \,                   
 h_2 =  h_{10},  \,                  
 c_1 =  f_{01} e_{01}, \,       
 c_2 =  f_{10} e_{10},  \,       
 c_3 =  f_{11} e_{11},   \,    
 c_4 =  f_{11} e_{10} e_{01},  \,
 c_5 =  f_{01} f_{10} e_{11}.                         
$$

Define the order on the set of perfect monomials: $$
 h_1  <   h_2 <  c_1  <  c_2  <  c_3  <  c_4  <  c_5\, .$$                  

Note that any monomial containing both variables $c_4$ and $c_5$ is not semi-perfect 
since
$m' = c_1 c_2 c_3$ has the same associated list of roots as 
$m = c_4 c_5$, but 
$c_1 < c_4$. We easily have from the relations between the generators the next proposition.

\begin{proposition}\label{thm_41}
	\
\begin{itemize}
\item The following set of monomials is a basis of 
$\hat U_0(\mathfrak{g})$: 
$$
P = \{ h_1^{s_1} h_2^{s_2} c_1^{s_3}  c_2^{s_4}  c_3^{s_5}  c_4^{s_6}\, | \, s_1, \ldots, s_6 \in \mathbb{N} \} \cup
    \{ h_1^{s_1} h_2^{s_2} c_1^{s_3}  c_2^{s_4}  c_3^{s_5}  c_5^{s_6} \, | \, s_1, \ldots, s_6 \in \mathbb{N} \}\, .   
$$
\item
The set  $\tilde{K}=\{ X\,  | \, X \in K',\,  {\rm Len}(X) = 2 \}$
   is a generating set of the ideal of relations $K$.
\end{itemize}
\end{proposition}

 The following $\tilde{K}$ is  the list of all  relations of length two:
\begin{eqnarray}
\label{eq:refer017a}
c_j h_i    & = & h_i c_j\, , \mbox{ for } i=1,2, j=1, \ldots, 5  \\
c_{2} c_{1} & = & -c_{5} + c_{4} + c_{1} c_{2}, \label{eq:refer017b}\\ 
c_{3} c_{1} & = &  c_{5} - c_{4} + c_{1} c_{3}, \\      
c_{4} c_{1} & = & -2 c_{5} +  (2 + h_{1}) c_{4} + c_{1}  c_{4} - c_{1}  c_{3} + c_{1} c_{2}, \\
c_{5} c_{1} & = & -h_{1}c_{5} + c_{1}  c_{5} + c_{1}  c_{3}  - c_{1}  c_{2}, \label{eq:refer017x} \\
c_{3} c_{2} & = & -c_{5} + c_{4} + c_{2}  c_{3}, \\
c_{4} c_{2} & = &  h_{2} c_{3}  +  h_{2} c_{4} + c_{2}  c_{4}     + c_{2}  c_{3} - c_{1}  c_{2}, \\
c_{5} c_{2} & = & -h_{2}c_{3}  - h_{2} c_{5} + c_{2}c_{5} - c_{2}c_{3} + c_{1}c_{2}, \label{eq:refer017y}\\  
c_{4} c_{3} & = & 2 c_{5}  - (h_{2} + h_{1} + 2)c_{4}    
                + c_{3}  c_{4} - h_{2}c_{3}  - c_{2}c_{3}  + c_{1}  c_{3}, \\
c_{5} c_{3} & = & (h_{2} + h_{1})c_{5} + c_{3}  c_{5}       + h_{2}c_{3}  + c_{2}  c_{3} - c_{1}  c_{3}, \\
c_{5} c_{4} & = & -(2 h_{2} + h_{1})c_{5} - c_{3}  c_{5}      
                - 2h_{2} c_{3}  + c_{2}  c_{5} - 2 c_{2}  c_{3} - c_{1}  c_{5}    
                + c_{1}  c_{2}  c_{3}  \nonumber \\
                && +   ( h_{2} + h_{1} + 2)c_{1} c_{2},           \label{eq:refer017z}\\
c_{4} c_{5} & = & -h_{1}c_{5}  - c_{3}  c_{5} +  h_{1} h_{2}c_{3}    + c_{2}  c_{5} 
               + h_{1} c_{2}  c_{3}  - c_{1}  c_{5}    +  h_{2} c_{1}  c_{3} + c_{1}  c_{2}  c_{3}.        
\label{eq:refer017}
\end{eqnarray}

The following Casimir elements generate the center of the universal enveloping algebra \cite{F1989}: 
\begin{eqnarray}
z_{1} & = &
c_{3} +c_{2}+c_{1}
+ {\frac{1}{3}} ( h_2^2
+ 3h_2
+ h_1^2
+ 3h_1
+ h_2 h_1)    \label{eq:refer017c}  \\
z_{2} & = &
c_{5}+c_{4}
+ {\frac{1}{3}}   (h_1-h_2)          c_{3}
- {\frac{1}{3}}   (6 +2 h_1 + h_2)   c_{2}
+ {\frac{1}{3}}   (h_1+2 h_2)        c_{1}  \nonumber \\
   && + {\frac{1}{27}} (-h_2-3+h_1) (6+2h_1+h_2) (h_1+2 h_2)
\label{eq:refer018}
\end{eqnarray}







Using \eqref{eq:refer017a}-\eqref{eq:refer018}, we can
reduce the number of generators of $U_0(\mathfrak{g})$ since by applying  \eqref{eq:refer017b},   \eqref{eq:refer017c}, and  
\eqref{eq:refer018} we can exclude  $c_3, c_4, c_5$ from all other relations.
As a result we will get a generating set $S = \{ h_1, h_2, z_1, z_2, c_1, c_2 \}$ of $U_0(\mathfrak{g})$ 
with the following decomposition:
  $S_1 = \{h_1, h_2, z_1, z_2\}$,  $S_2 = \{c_1\}$, $S_3 = \{c_2\}$.
  For the set of relations we have
   $R = R_1 \cup R_2$, where $R_1$ consists of one relation obtained from 
    \eqref{eq:refer017x}, and 
    $R_2$ consists of two relations obtained from 
     \eqref{eq:refer017y} and   \eqref{eq:refer017z}.


\subsection{Generic torsion free $A_2$-modules}
\label{sec:A2TorFreeSec}

Let $\Gamma$ be the commutative subalgebra of $U_0(\mathfrak{g})$ generated  by the elements
$h_1$, $h_2$, $z_1$, $z_2$, $c_1$. This is a Gelfand-Tsetlin subalgebra of $\mathfrak{sl}(3)$.
We give a construction of a family of $\Gamma$-pointed modules $V(a_1, a_2, a_3, \xi, \mu)$ 
which depends on five complex parameters with the restriction $a_3\notin \mathbb Z$. 
These are essentially the universal generic  Gelfand-Tsetlin modules initially constructed 
in \cite{DFO1984} and \cite{FGR2021} using a Gelfand-Tsetlin basis \cite{GT1950}. 

Let us fix arbitrary $a_1, a_2, a_3, \xi, \mu\in \mathbb C$, such that $a_3\notin \mathbb Z$. 
  To simplify the formulas we define the following set of indexed variables: 
\begin{equation}
\begin{aligned}   
 h_{ij}^{(1)}  &= a_1 + 2i -  j,   \,\,
 h_{ij}^{(2)}   = a_2 -  i + 2j,    \,\,
 s_{j,k}        = a_3 -  j + 2k - 1,     \\
 S^{+}_{ijk}   &= \frac{1}{2}( s_{j,k} + h_{i,j}^{(1)} ) \label{eq:REF109}
                = \frac{1}{2}( a_1  + a_3 -1 )     + i-j+k,           \\
 S^{-}_{ik}    &= \frac{1}{2}( s_{0,k} -  h_{i,0}^{(1)})                 \,
                = \frac{1}{2}(-a_1  + a_3 -1 ) - i+k,                 \\
 T^{+}_{k}     &= \frac{1}{2}( s_{0,k} + \frac{1}{3} ( h_{0,0}^{(1)}+ 2 h_{0,0}^{(2)}))  \, 
                = \frac{1}{6}( a_1+2a_2+3a_3) + k -  \frac{1}{2},                     \\
 T^{-}_{jk}    &= \frac{1}{2}( s_{j,k} - \frac{1}{3} ( h_{0,j}^{(1)}+ 2 h_{0,j}^{(2)}))  \,
                = \frac{1}{6}(-a_1-2a_2+3a_3) -j + k - \frac{1}{2},                   \\
Q^{-}_{jk}     &= -\mu  +   T^{-}_{j,k-2} \xi -  T^{-}_{j,k} T^{-}_{j,k-1} T^{-}_{j,k-2}
                = -(T^{-}_{j, k-1} - t_1 )(T^{-}_{j, k-1}  -  t_2 ) (T^{-}_{j, k-1} - t_3),  \\
 Q^{+}_{k}     &=  \mu  +   T^{+}_{k}     \xi -  T^{+}_{k}   T^{+}_{k-1}   T^{+}_{k-2}
                = (T^{+}_{k-1} + t_1 )(T^{+}_{k-1}  +  t_2 ) (T^{+}_{k-1} + t_3),  \\
 & \text{where}\,\,  t_1, t_2, t_3 \,\, \text{are three roots of the equation}\,\, t^3 -t(\xi+1) +  \mu + \xi = 0\, .
\end{aligned}
\end{equation}

Next, we define an action of the Lie algebra $\mathfrak{g}$ on the vector space  
$V(a_1, a_2, a_3, \xi, \mu)= \{{\bf v}_{ijk}\, |\, i,j,k \in \mathbb Z\}$ as follows:
\begin{equation}
\begin{aligned} \label{eq:REF_A2418}
z_{1}({\bf v}_{ijk})    &= \xi {\bf v}_{ijk},               \,\,
z_{2}({\bf v}_{ijk})     = \mu {\bf v}_{ijk},              \,\,
h_{01}({\bf v}_{ijk})    = h_{ij}^{(1)}   {\bf v}_{ijk},    \,\,
h_{10}({\bf v}_{ijk})    = h_{ij}^{(2)}   {\bf v}_{ijk},    \\
e_{01}({\bf v}_{ijk})   &= S^{+}_{ijk}    {\bf v}_{i+1,j,k}, \,\,       
f_{01}({\bf v}_{ijk})    = S^{-}_{ik}     {\bf v}_{i-1,j,k}, \\                     
e_{10}({\bf v}_{ijk})   &= Q^{-}_{jk}     {\bf v}_{i,j+1,k}   
                    + {\frac{S^{-}_{i,k}}{ s_{j+1,k}   s_{jk} }}  {\bf v}_{i,j+1,k+1}, \\
f_{10}({\bf v}_{ijk})   &= Q^{+}_{k}      {\bf v}_{i,j-1,k-1} 
                    + {\frac{S^{+}_{i,j,k}}{ s_{j+1,k} s_{jk} }}  {\bf v}_{i,j-1,k}. 
\end{aligned}
\end{equation}

We have the following statement (cf. \cite{F1989,FGR2021}).

\medskip

\begin{theorem}\label{prp:prp1} Let $a_1, a_2, a_3, \xi, \mu\in \mathbb C$ and $a_3\notin \mathbb Z$. 
Then
\begin{itemize}
\item[1.]  The space   $V(a_1, a_2, a_3, \xi, \mu)$  is a torsion free 
    $\Gamma$-pointed $\mathfrak{g}$-module if and only if  $S^{+}_{ijk}$ and $S^{-}_{ik}$ are nonzero,  for all $i,j,k  \in \mathbb Z$. 
\item[2.]   The module $V(a_1, a_2, a_3, \xi, \mu)$ is simple if and only if
      $S^{+}_{ijk}$, $S^{-}_{ik}$, $Q^{-}_{jk}$, $Q^{+}_{k}$ are nonzero,  for all $i,j,k  \in \mathbb Z$. 
\end{itemize}
\end{theorem}

\begin{proof} The fact that  $V(a_1, a_2, a_3, \xi, \mu)$  is 
    a   $\mathfrak{g}$-module was shown in \cite{F1989,FGR2021}. It is 
     $\Gamma$-pointed as $\Gamma$ separates the basis elements. It follows from formulas above that the module  $V(a_1, a_2, a_3, \xi, \mu)$ is torsion free if and only if $S^{+}_{ijk}$ and $S^{-}_{ik}$ are different 
      from zero  for all $i,j,k  \in \mathbb Z$. 
         Note that for  every $\mathfrak{h}$-weight $\lambda$ of $V(a_1, a_2, a_3, \xi, \mu)$, the $\lambda$-weight subspace is infinite dimensional with basis ${\bf v}_{ijk}$, where $i, j$ are determined by $\lambda$ and $k$ runs through $\mathbb Z$. In this basis the operator $c_1$ is presented by an infinite diagonal matrix and the operator $c_2$ is presented by   a  3-diagonal matrix $(b_{st})$ with $b_{st}=0$ for all $s,t$ such that $|s-t| > 1$.  This is a consequence of  relations above. 
The simplicity of $V(a_1, a_2, a_3, \xi, \mu)$ results in the simplicity of each $\lambda$-weight subspace as $U_0(\mathfrak{g})$-module. This implies the required conditions in item 2.
\end{proof}

\medskip

Simple subquotients of the module  $V(a_1, a_2, a_3, \xi, \mu)$ give all simple generic Gelfand-Tsetlin $\mathfrak{g}$-modules with finite or infinite weight multiplicities \cite{FGR2015}. Moreover, we have an immediate consequence of the previous theorem.


\begin{corollary} \label{prp:prp1}
 If $V'$ is a simple generic Gelfand-Tsetlin $\mathfrak{g}$-module, then it is isomorphic to a subquotient of $V(a_1, a_2, a_3, \xi, \mu)$ 
   for some suitable parameters such that $0 \le Re \, a_1 < 1$, $0 \le Re \, a_2 < 3$, $0 < Re \, a_3 < 2$, 
        $a_3\neq 1$, where $Re \, a$ stands for the real part of $a$.
\end{corollary}


\subsection{ Subquotients  of $V(a_1, a_2, a_3, \xi, \mu)$}

As we have just seen, the 
$\Gamma$-pointed module 
$V(a_1, a_2, a_3, \xi, \mu)$ is simple when 
$S^{+}_{ijk}$, $S^{-}_{ik}$, $Q^{-}_{jk}$, $Q^{+}_{k}$ are nonzero,  for all $i,j,k  \in \mathbb Z$. Suppose now that one or more coefficients are equal to zero. Let 
$I = \mathbb Z^3 \subset \mathbb R^3$ 
be the weight lattice (the lattice of indexes) of the module 
$V(a_1, a_2, a_3, \xi, \mu)$.
We then have the following cases:\\

{\bf Case 1}. Assume that $S^{+}_{i_0,j_0,k_0} = 0$, for some $i_0,j_0,k_0 \in \mathbb Z$. Then
$a_1 + a_3 \in 2\mathbb Z + 1$. 
Consider the following function $F(i,j,k) = S^{+}_{ijk}$ on the set $I$. 
The subset of zero points $P(F) = \{(i,j,k) \, |\,  F(i,j,k)=0\}$ is called the splitting hyperplane
for the function $F$.
The splitting hyperplane $F(i,j,k)$ divides the set $I$ in two subsets, called the positive and the negative
components, $I =I_F^{+}\cup I_F^{-}$ as follows:
$$
    I_F^{+}  =\{(i,j,k) \, | \, F(i,j,k) > 0\} \quad \mbox{and}\quad
    I_F^{-}  =\{(i,j,k) \, | \, F(i,j,k)\le 0\}\, .
$$

Then the subspace 
$V^\prime = \{{\bf v}_{ijk}\, | \, (i,j,k)  \in I_F^{-} \} \subset V$
is a simple submodule of $V$.
The same holds if $a_1 - a_3 \in 2\mathbb Z+1$ and $S^{-}_{i_0,k_0} = 0$, for some $i_0,k_0 \in \mathbb Z$.

\medskip

{\bf Case 2.}
Assume that $Q^{+}_{k} = 0$, for some $k \in \mathbb Z$.
As we can see, $Q^{+}_{k}$ is the product of three first degree polynomials $F_m(i,j,k)=(T^{+}_{k}-t_m)$, for $m=1,2,3$ in the variable $k$.

Let us consider the case when $Q^{+}_{k} = 0$ has three distinct roots. We need to find possible values of the parameters
$a_1,a_2,a_3,t_1,t_2,t_3$ such that the system of three equations $F_m(i,j,k_m)=0$, with $m=1,2,3$,  has integer solutions for the variables $k_m$.
We find that $Q^{+}_{k} = 0$, for  $k=k_1,k_2,k_3 \in \mathbb Z$, if 
 the following conditions hold: $t_m = \frac{1}{3}(k_1+k_2+k_3) - k_m$, for $m=1,2,3$, and $a_3 = 3-\frac{1}{3}(2k_1+2k_2+2k_3+a_1+2a_2)$.
Suppose that $k_1 > k_2 > k_3$.
Then we have three splitting hyperplanes  $P_m=\{ (i,j,k)\, | \, F_m(i,j,k) =0 \}$ and $I =I_{F_m}^{+}\cup I_{F_m}^{-}$, where
$$
    I_{F_m}^{+}  =\{(i,j,k)\,  |\,  F_m(i,j,k) \ge 0\}\quad \mbox{and}\quad
    I_{F_m}^{-}  =\{(i,j,k)\,  |\,  F_m(i,j,k) < 0\}$$ for $m=1,2,3$.
Consequently, we obtain:

\begin{proposition}\label{lem:lem101c} 
The subspaces 
  $V_m^\prime = \{{\bf v}_{ijk}\, | \; (i,j,k)  \in I_{F_m}^{+} \} \subset V(a_1, a_2, a_3, \xi, \mu)$
are  submodules of $V(a_1, a_2, a_3, \xi, \mu)$
satisfying the inclusion 
  $V_1^\prime  \subset V_2^\prime   \subset  V_3^\prime$.

If none of $S^{+}_{ijk}$ or $S^{-}_{ik}$ or $Q^{-}_{jk}$  are zero, for all choices of indices, then  
$V_1^\prime$ and  
$V_2^\prime/V_1^\prime$ and
$V_3^\prime/V_2^\prime$ and
$V(a_1, a_2, a_3, \xi, \mu)/V_3^\prime$ 
are simple torsion-free modules.
The dimensions of the corresponding weight spaces are 
$V_1^\prime$: $\infty$,      
$V_2^\prime/V_1^\prime$: $k_1-k_2$, $V_3^\prime/V_2^\prime$: $k_2-k_3$, and 
$V(a_1, a_2, a_3, \xi, \mu)/V_3^\prime$: $\infty$.
\end{proposition}

A similar construction of a family of torsion-free submodules can be obtained in the case when $Q^{-}_{j,k}=0$. 
Combining the conditions under which one or more coefficients are equal to zero we can obtain several distinct modules.
The full classification of $\mathfrak{sl}(3)$ modules was obtained in \cite{FGR2021} using the Gelfand-Tsetlin tableaux technique.


\section{Construction of simple weight $C_2$-modules}

In this section we use the technique developed in the previous section to construct 
simple Gelfand-Tsetlin modules for the Lie algebra of type $C_2$.

\subsection{The centralizer for $C_2$}
\label{sec:Centrlzr}

Consider  the root system of $C_2$:
$$\Delta  =  \{ \alpha_1, \alpha_2, \alpha_3= \alpha_1+\alpha_2, \alpha_4=2\alpha_1+\alpha_2,  
  \alpha_5=-\alpha_4,  \alpha_6=-\alpha_3, \alpha_7=-\alpha_2, \alpha_8=-\alpha_1\}$$  
and the Chevalley basis:
$e_{10} = E_{12} - E_{43}$,                                               
$f_{10} = E_{21} - E_{34}$,                                   
$e_{01} = E_{31}$,                                                         
$f_{01} = E_{13}$,                                             
$h_{10}  = ( E_{11} - E_{22}- E_{33} + E_{44}), $
$h_{01} = -E_{11} + E_{33},              $
$e_{11}  = -E_{32} - E_{41}$,                     
$e_{21}  = 2 E_{42}$,                     
$f_{11}= -E_{14} - E_{23}$,                     
$f_{21} = 2 E_{24}$                   
 with the order
$
h_{10} < h_{01} < f_{10} < f_{01} < f_{11} < f_{21} < e_{21}
 < e_{11} < e_{01} < e_{10}. 
$

 
The following is a complete set of indecomposable lists of roots:
$\{\alpha_1, \alpha_8\}$,          
 $\{\alpha_2, \alpha_7\}$,        
 $\{\alpha_3, \alpha_6\}$,       
 $\{\alpha_4, \alpha_5\}$,          
$ \{\alpha_1, \alpha_2, \alpha_6\}$,       
$ \{\alpha_3, \alpha_7, \alpha_8\}$,       
$ \{\alpha_1, \alpha_3, \alpha_5\}$,       
$ \{\alpha_4, \alpha_6, \alpha_8\}$,             
 $\{\alpha_1, \alpha_1, \alpha_2, \alpha_5\}$,  
 $\{\alpha_4, \alpha_7, \alpha_8, \alpha_8\}$,  
$ \{\alpha_3, \alpha_3, \alpha_5, \alpha_7\}$,  
 $\{\alpha_2, \alpha_4, \alpha_6, \alpha_6\}$.

Choose the following generators of $U_0=U_0(C_2)$:
$$
  h_{1}  = h_{01}, \,
   h_{2}  = h_{10}, \,
  c_{1}  = f_{01} e_{01}, \,
  c_{2}  = f_{21} e_{21}, \,
  c_{3}  = f_{10} e_{10}, \,
  c_{4}  = f_{11} e_{11}, $$
$$  c_{5}  = f_{11} e_{01} e_{10}, \,
  c_{6}  = f_{10} f_{01} e_{11}, \,
  c_{7}  = f_{21} e_{11} e_{10}, \,
  c_{8}  = f_{10} f_{11} e_{21}, $$
$$  c_{9}  = f_{21} e_{01} e_{10} e_{10}, \,
  c_{10}  = f_{10}f_{10} f_{01} e_{21}, \,
  c_{11}  = f_{01}f_{21} e_{11} e_{11}, \,
  c_{12}  = f_{11} f_{11} e_{21} e_{01}.  $$

 Order the set of perfect monomials as follows:
$$
  h_{1} < h_{2} < c_1 < c_2 < c_3 < c_4 < c_5 < c_6 < c_7 < c_8 < c_9 < c_{10} < c_{11} < c_{12}\, .                          
$$

As in type $A_2$ we have:
\begin{proposition}\label{lem:lemmaC2_51}
	\
\begin{itemize}
\item[1)] The following set of monomials is a basis of 
$U_0$: 
\begin{align}
P = & \{  h_1^{s_1}   h_2^{s_2} c_1^{s_3}  c_2^{s_4}  c_3^{s_5}  c_4^{s_6}  c_m^{s_m}  c_n^{s_n}\,  |
 \, s_1,s_2, \ldots, s_6,s_m,s_n \in \mathbb{N} \text{ and } \nonumber\\
    &  (m,n) \in \{ (5,9), (5,12),(6,10),(6,11),(7,9),(7,11),(8,10),(8,12) \}  \}.   \nonumber
\end{align}

\item[2)] 
The set  $\tilde{K}=\{ X\,  |\,  X \in K', {\rm Len}(X) = 2 \}$
   is a generating set of the ideal of relations $K$. 
\end{itemize}
\end{proposition}

\begin{proof}
Suppose that $X$ is a semi-perfect monomial. Then it cannot contain certain pairs of perfect elements.
For example, $X$ cannot contain both $c_5$ and $c_6$, since $L_1(c_5 c_6) = L_1(c_1 c_3 c_4)$
and $c_1 c_3 c_4 < c_5 c_6$.
 A direct calculation shows that the only possible combinations of perfect monomials 
 are the ones given above, proving 1).
 Recall now the  relations  $c_i c_j = \mathcal{N}( c_i c_j )$, for $i,j = 1, \ldots, 12$. 
One can easily see that   $\mathcal{N}( X_1 X_2)$ can be computed using only relations of length two
for any semi-perfect monomials $X_1$, $X_2$. For example, 
 if  $X_1 = c_5 c_9$ and $X_2 = c_8  c_{12}$, then
$Y = c_2 c_3 c_5 c_5 c_{12}$ is semi-perfect 
and $L_1(X_1 X_2) = L_1(Y)$. Moreover, $L_1(c_8 c_9) = L_1(c_2 c_3 c_5)$.
\end{proof}

Denote $h_3 = h_1+h_2$ and let $U^\prime_0 = U_0 [h^{-1}_1, h^{-1}_3 ]$. It will be convenient here to work with $U^\prime_0$.

The list of all relations is rather big, so we will give only those that are used in our calculations.
They are used to express  $c_{12}, \ldots, c_5$ via
  $c_{4}$, $c_3$, $c_2$, $c_1$, $h_1$, $h_3$, $z_1$ in $U^\prime_0$:

\begin{align}  
\label{en:refer0209}
c_{12}  = & -c_{10} - c_{2} - 2 c_{8} + [ c_{1}, c_{8} ],   \\
c_{11}  = & -c_{9} -2c_{7} +c_{2}- [c_{1},c_{7}],   \\
c_{10}  = & c_{9} - c_{8} + c_{7}  + {\frac{1}{2}}  [c_{5}, c_{2}]  + [c_{1}, c_{8}]\\
c_{ 9}  = & {\frac{1}{2 h_1}}  ([c_1,[c_1,c_7]] 
 - 2 c_{ 1}c_{ 2}  
 - 2 c_{ 7}c_{ 1} 
 - 2 c_{ 1}c_{ 7} )
 +   {\frac{1}{2}} [ c_7,c_1]  
 - 2 c_{ 7} 
 -   c_{ 2}  ,\\
c_{8}  =  &   c_{7}  +{\frac{1}{2}}[c_{2}, c_{3}]                         \\ 
c_{ 7}  = & {\frac{1}{16 h_3}}  (-[c_2,[c_2,c_3]] 
 + 8 c_{ 3}c_{ 2}  
 - 8 c_{ 2}c_{ 4} 
 - 8 h_{ 1}c_{ 2} )
 + {\frac{1}{4}} [ c_3,c_2]  
 - {\frac{1}{2}} c_{ 2}  ,\\
c_{6}  =  &   c_{5}   + [c_{3}, c_{1}]                                     \\
c_{5}  =  & - c_{4} + c_1c_3 + {\frac{1}{2h_{1}} } ( -[c_1, c_1, c_3 ] 
             - c_3c_{1} (h_{1}-2) - 2c_1c_4).                     \label{en:refer0209a}
\end{align}

The next  relations are used to find relations between the
elements  $c_{4}$, $c_3$, $c_2$, $c_1$, $h_1$, $h_3$, $z_1$.

\begin{align}  
c_5    c_1 & = -c_6 
              +(h_1 +1) c_5 
              +h_1  c_4 
              +c_1    c_5 
              +c_1    c_4 
              -c_1    c_3    \label{en:refer0219a}  \\       
c_7    c_2 & = -4 h_3  c_7 
              +c_2    c_7 
              -2 c_2    c_4 
              +2 c_2    c_3 
              +(-2 h_3 -2 h_1 ) c_2  \label{en:refer0219b}\\
c_5    c_3 & = c_9 
             +c_8 -2 c_6 
             +(h_3 -h_1 ) c_5 
             +c_3    c_5 
             -c_3    c_4 
             +2 c_1    c_3      \label{en:refer0219c}\\
c_7    c_3 & = -2 c_9 
              -2 c_8 
              +(h_3 -h_1 ) c_7 
              +4 c_6 
              +c_3    c_7 
              +2 c_3    c_4 
              -c_2    c_3       \label{en:refera0219d}\\
c_6    c_5 & = c_{10} 
             +2 c_8 
             -c_7 
             +(h_3 -h_1 -2) c_6 
             -c_4    c_5              \nonumber \\
          &  -c_3    c_6 
             -2 c_3  c_4 
             -c_1    c_8 
             +2 c_1  c_6 
             +c_1    c_3    c_4 
             +(h_3 +h_1 +2) c_1    c_3.     \label{en:refera0219e}
\end{align}          
                                 


The center of the universal enveloping algebra of $\mathfrak{g}$ is generated by the  Casimir elements
$$z_1 =  4c_1 +c_2 +2c_3 +2c_4 + 2 h_1^2 + 2 h_3^2+2h_1 +4h_3$$  and
\begin{align*}  
z_2 = 
&
 2c_{12}
+2c_{11}
-2c_{10}
-2c_{9}
+(2h_{1}+1)c_{8}
+(2h_{1}-1)c_{7}
\nonumber\\
&
 + (4h_{3}+6) c_{6}
 + (4h_{3}+10) c_{5} 
 - c_{4}^2
 + (-2h_{1}h_{3}-4h_{1}+2h_{3}+6) c_{4} 
 - 2c_{3}c_{4}
\nonumber\\
&
 - c_{3}^2
 + (2h_{1}h_{3}+4h_{1}+2h_{3}+6) c_{3}
 - (h_{1}-1)(h_{1}+1) c_{2}
 - 4c_{1}c_{2}
\nonumber\\
&
-4(h_{2}+3)(h_{2}+1) c_{1}
-h_{1}(h_{3}+3)(h_{3}+1)(h_{1}+2).
\end{align*}  

As in type $A$, our goal is to find the ``best" choice of a generating set $S$ of $U_0$, such that 
the cardinality of the set $S_3$ is minimal.


\begin{lemma}\label{lem: lemmaC2_54}
	\
\begin{itemize}
\item[1.] The set $S = \{ h_1, h_2, z_1,z_2, c_1, c_2,c_3 \}$ is a generating set of 
  the centralizer 
 $U^\prime_0$  with the following decomposition:
  $S_1 = \{h_1, h_2, z_1, z_2\}$,  $S_2 = \{c_1, c_2\}$, $S_3 = \{c_3\}$.

\item[2.] 
  The decomposition of the set of relations is the following:
   $R = R_1 \cup R_2$,  where $R_1$ consists of two relations obtained from 
     \eqref{en:refer0219a}-\eqref{en:refer0219b}, while 
    $R_2$ consists of three relations obtained from 
     \eqref{en:refer0219c}-\eqref{en:refera0219e}.
\end{itemize}    
\end{lemma}

\begin{proof}
Using \eqref{en:refer0209}-\eqref{en:refer0209a}, we can exclude the generators $c_{12}, \ldots, c_4$ from all other relations.
As the result, we will get the generating set $S= \{ h_1, h_2, z_1, c_1, c_2, c_3 \}$. The rest can be verified by direct computations.
\end{proof}

\begin{remark}
We cannot claim that $S$ is a generating set of $U_0$. Nevertheless,  any $U_0$-module $M$ with a nonzero scalar action of $h_1$ is a $U^\prime_0$-module.
\end{remark}

\subsection{Construction of torsion free $C_2$-modules}
\label{sec:C2torsionfree}

Let $\Gamma$ be the commutative subalgebra of $U_0(C_2)$ generated  by the elements
$h_1, h_2, z_1,z_2, c_1$. 
We give a construction of two families of $\Gamma$-pointed modules which depend on $4$ complex parameters.  

Fix arbitrary complex $a_1, a_2, a_3, a_4$ and $\xi$ and define the following set of indexed variables: 
\begin{equation}
\begin{aligned}
  h_{ij}^{(1)}  &  = a_1 + 2i - j, \,\,
  h_{ij}^{(2)}    = a_2 - 2i +2j, \,\,
  s_{jk}          = a_3-j+2k-1,   \,\,
  Q^{\pm}_{jk}  = \frac{\upsilon}{s_{jk}} \pm 1,                        \\                         
  S^{+}_{ijk}  &  = \frac{1}{2}(   a_1 + a_3 + 2 i - 2 j + 2 k -1 ), \,\,
  S^{-}_{ik}     = \frac{1}{2}(  -a_1 + a_3 - 2 i       + 2 k -1 ), \\
  T^{+}_{k}    & = \frac{1}{2}(   a_1 + a_2 + a_4       + 2 k -1),  \,\,                           
  T^{-}_{jk}     = \frac{1}{2}(  -a_1 - a_2 + a_4 - 2 j + 2 k -1),  \\
               & \text{ where}  \,  i, j, k \in \mathbb Z, \,\, \text{ and }\,\, \upsilon \,\, \text{is a root of the equation}\,\, \xi = 2(\upsilon+1)(\upsilon-2).                          
\label{eq:refer028}
\end{aligned}
\end{equation}

 Consider the vector space  $V(a_1, a_2, a_3, a_4, \xi)= \{{\bf v}_{ijk} \, |\,  i,j,k \in \mathbb Z\}$ and define the following operators on $V(a_1, a_2, a_3, a_4, \xi)$ (abusing the notation and using same letters as the generators of 
  $\mathfrak{g}$):
\begin{equation}
\begin{aligned}
\label{eq:REF103qq1}
h_{1}({\bf v}_{ijk})   & = h_{ij}^{(1)} {\bf v}_{ijk},  \,\,
h_{2}({\bf v}_{ijk})     = h_{ij}^{(2)} {\bf v}_{ijk},   \\
e_{01}({\bf v}_{ijk})  &  = S^{+}_{ijk}  {\bf v}_{i+1,j,k},  \,\,
f_{01}({\bf v}_{ijk})    = S^{-}_{ik}   {\bf v}_{i-1,j,k},    \\  
e_{10}({\bf v}_{ijk})  & = S^{-}_{ik}      Q^{+}_{j+1,k}    {\bf v}_{i,j+1,k+1}
                   +  T^{-}_{j+1,k}   Q^{-}_{j+1,k}    {\bf v}_{i,j+1,k}, \\     
f_{10}({\bf v}_{ijk})  & =  S^{+}_{ijk}     Q^{+}_{j+1,k}    {\bf v}_{i,j-1,k}  
                   +  T^{+}_{k-1}     Q^{-}_{j+1,k}    {\bf v}_{i,j-1,k-1},\\
e_{11}({\bf v}_{ijk})  &= - S^{+}_{ijk}   Q^{+}_{j+1,k}   {\bf v}_{i+1,j+1,k+1}               
                     + T^{-}_{j+1,k} Q^{-}_{j+1,k}   {\bf v}_{i+1,j+1,k},\\
f_{11}({\bf v}_{ijk})  &=   S^{-}_{i,k}   Q^{+}_{j+1,k}   {\bf v}_{i-1,j-1,k} 
                     - T^{+}_{k-1}   Q^{-}_{j+1,k}   {\bf v}_{i-1,j-1,k-1},\\
e_{21}({\bf v}_{ijk})  &= 2 T^{-}_{j+1,k}                 {\bf v}_{i+1,j+2,k+1}, \,\,
f_{21}({\bf v}_{ijk})  = 2 T^{+}_{k-1}                   {\bf v}_{i-1,j-2,k-1}.                  
\end{aligned}
\end{equation}

These formulas define a $\Gamma$-module structure on $V(a_1, a_2, a_3, a_4, \xi)$, but at this point we cannot claim a $\mathfrak{g}$-module structure.

\begin{lemma}\label{ lemmaC2_55-gamma-C2}
The subalgebra $\Gamma$ has a simple spectrum on $V(a_1, a_2, a_3, a_4, \xi)$, and hence
separates the basis elements  ${\bf v}_{ijk}$ if and only if $a_3\notin \mathbb Z$. 
\end{lemma}

\begin{proof} We need to show that $\Gamma$ acts with different characters on the basis elements ${\bf v}_{i,j,k}$. It is sufficient to consider vectors from the same weight space. Suppose
$$h_{1}({\bf v}_{ijk}) = (a_1 + 2i - j){\bf v}_{ijk}=\lambda {\bf v}_{ijk},\,
  h_{2}({\bf v}_{ijk}) = (a_2 - 2i +2j){\bf v}_{ijk}=\mu {\bf v}_{ijk}\, ,$$
for some fixed $\lambda$ and $\mu$. Then $i$ and $j$ are uniquely determined: $j=\lambda+\mu-a_1-a_2$ and $i=\frac{1}{2}(\lambda+j-a_1)$. Thus, the basis elements of this weight subspace differ by the third index. 
Consider 
${\bf v}_{ijk_1}$ and ${\bf v}_{ijk_2}$, for arbitrary integers $i,j,k_1,k_2$. We have

\begin{eqnarray*}
	e_{01}f_{01}({\bf v}_{ijk}) & = & S^{-}_{ik} S^{+}_{i-1,j,k}({\bf v}_{ijk}) \\
	&=&\frac{1}{4}(  -a_1 + a_3 - 2 i+ 2 k -1 )(   a_1 + a_3 + 2 i - 2 j + 2 k -3 ){\bf v}_{ijk}\\
	& = & (4k^2+4(a_3-j-2)k+( -a_1 + a_3 - 2 i-1)( a_1 + a_3 + 2 i - 2 j -3)){\bf v}_{ijk}\, .\\	
\end{eqnarray*} 
 Suppose that $c_1$ has the same value on ${\bf v}_{ijk_1}$ and ${\bf v}_{ijk_2}$.
 Then 
$$4k_1^2+4(a_3-j-2)k_1=4k_2^2+4(a_3-j-2)k_2$$ and $$k_1+k_2=-(a_3-j-2)\, , $$ which implies the statement.
\end{proof}

Denote $V_1(a_1, a_2, a_3, \xi)=V(a_1, a_2, a_3, a_3, \xi)$ and $V_2(a_1, a_2, a_3, a_4)=V(a_1, a_2, a_3, a_4, -4)$.

\medskip

\begin{theorem}\label{prop:Prop21-C2} Let $V=V(a_1, a_2, a_3, a_4, \xi)$. Then
\begin{itemize}
\item[1.] $V$ is a   $\mathfrak{g}$-module  if and only if  $V=V_1(a_1, a_2, a_3, \xi)$, with $a_3 \notin \mathbb Z$,  or $V=V_2(a_1, a_2, a_3, a_4, -4)$.

\item[2.] The $\mathfrak{g}$-module $V$ is a torsion free $\Gamma$-pointed  $\mathfrak{g}$-module 
 if and only if  $S^{+}_{ijk}S^{-}_{ik}T^{-}_{jk}T^+_{k} \neq 0$, for all $i,j,k  \in \mathbb Z$. 

\item[3.] The torsion free $\Gamma$-pointed $\mathfrak{g}$-module $V$ is   simple if and only if
      $Q^{+}_{jk}Q^{-}_{jk} \neq 0$, for all $ j, k  \in \mathbb Z$.       
      
\item[4.] If $V'$ is a simple torsion free $\Gamma$-pointed  $\mathfrak{g}$-module with a basis parametrized by the lattice $\mathbb Z^3$ and with separating  action of $\Gamma$ on basis elements, 
then  $V'\simeq V(a'_1, a'_2, a'_3, a'_4, \xi')$, 
    for some suitable choice of parameters such that
   $0 \le Re \, a'_1 < 1$, 
   $0 \le Re \, a'_2 < 2$, 
   $0 <   Re \, a'_3 < 2$,  
  where $Re \, a$ denotes the real part of $a$.

\item[5.] The action of the Casimir elements provides, in the case $a_3 = a_4$,
$$
z_{1}({\bf v}_{ijk}) =    \xi {\bf v}_{ijk} \quad \mbox{and} \quad
z_{2}({\bf v}_{ijk}) =   -\frac{1}{4} \xi(\xi+4) {\bf v}_{ijk}     $$
and, in the case $a_3 \neq a_4$, $$
z_{1}({\bf v}_{ijk}) =     ((a_3 - a_4)^2 - 4) {\bf v}_{ijk}\quad \mbox{and} \quad z_{2}({\bf v}_{ijk}) =  0\,. $$                   
\end{itemize}
\end{theorem}
\begin{proof} First we observe that \eqref{eq:refer028}-\eqref{eq:REF103qq1}  are well defined if and only if   $V=V_1(a_1, a_2, a_3, \xi)$, with $a_3 \notin \mathbb Z$,  or $V=V_2(a_1, a_2, a_3, a_4, -4)$. 
The claim that $V$ is a $\mathfrak{g}$-module follows by checking all defining relations of the Lie algebra. On the other hand \eqref{en:refer0219a}-\eqref{en:refera0219e} give only two alternatives 
$V_1(a_1, a_2, a_3, \xi)$ and $V_2(a_1, a_2, a_3, a_4)$.
It follows immediately from the formulas of the action of $\mathfrak{g}$ that $V(a_1, a_2, a_3, a_4, \xi)$ is
 a  torsion free  module 
 if and only if  $S^{+}_{i,j,k}$, $S^{-}_{i,k}$,  $T^{-}_{jk}$ and $T^+_{k-1}$ are different from zero, for all $i,j,k  \in \mathbb Z$. 
 Note that $\Gamma$ has a simple spectrum on $V(a_1, a_2, a_3, a_4, \xi)$ and hence, the action 
 of $\Gamma$ separates the basis elements by 
 Lemma \ref{ lemmaC2_55-gamma-C2}. In particular, $V(a_1, a_2, a_3, a_4, \xi)$ is $\Gamma$-pointed. 
 This implies the second statement.

 Suppose now that the module $V(a_1, a_2, a_3, a_4, \xi)$ is torsion free. Similarly to case of $A_2$, using  the relations in $U_0(\mathfrak{g})$ we get that $c_1, c_2$ are presented by infinite diagonal matrices, while
the element $c_3$ is presented by a $3$-diagonal matrix on every weight subspace of $V(a_1, a_2, a_3, a_4, \xi)$.
Using this fact and the separating action of $\Gamma$ on basis elements, it follows that the conditions 
      $Q^{+}_{jk}Q^{-}_{jk} \neq 0$, for all $ i,j,k  \in \mathbb Z$, are necessary and sufficient to guarantee that 
      any element of $V(a_1, a_2, a_3, a_4, \xi)$ generates the whole module, which is equivalent
       to the simplicity of the module. This implies the third statement.

 Let $V'$ be a simple torsion free $\Gamma$-pointed  $C_2$-module with a basis
  $\{v'_{ijk}, \, i,j,k \in \mathbb Z\}$ such that $\Gamma$ acts by different characters on the 
  basis elements $v'_{ijk}$. Fix one basis element  $v'_{ijk}$ and apply the generators of
   the centralizer $U_0(\mathfrak{g})$. One can see directly from the action that $U_0(\mathfrak{g})v'_{ijk}$ will 
   be equal to the whole weight space of $V'$, which $v'_{ijk}$ belongs to. Moreover, the action
    of the generators of $\Gamma$ will determine uniquely the corresponding parameters
     $a_1, a_2, a_3, a_4, \xi$. Hence, we get a nonzero homomorphism $\theta$ of $U_0(\mathfrak{g})$-modules 
     $U_0(\mathfrak{g})v_{ijk}$ and $U_0(\mathfrak{g})v'_{ijk}$: $\theta(v_{ijk})=v'_{ijk}$. Moreover, $U_0(\mathfrak{g})v'_{ijk}$ 
     is a simple $U_0(\mathfrak{g})$-module as $V'$ is simple. Hence, $\theta$ is surjective. 
     It extends to a surjective homomorphism 
 from $V(a_1, a_2, a_3, a_4, \xi)$ to $V'$. Comparing the bases of both modules we conclude 
 the isomorphism $V'\simeq V(a_1, a_2, a_3, a_4, \xi)$. This proves the fourth statement. 
 
 The last statement follows by direct computation.
\end{proof}

Therefore, Theorem \ref{prop:Prop21-C2} provides two $4$-parameter families of simple torsion free $\Gamma$-pointed  $\mathfrak{g}$-modules. If  $V(a_1, a_2, a_3, a_4, \xi)$ is a torsion free $\Gamma$-pointed  $\mathfrak{g}$-module which is not simple, then all its simple subquotients are torsion free $\Gamma$-pointed  modules. They exhaust all \emph{generic} simple torsion free $\Gamma$-pointed  $\mathfrak{g}$-modules, which are analogs of generic simple Gelfand-Tsetlin modules in type $A$.





     

\section{Gelfand-Tsetlin modules for $G_2$} \label{sec6}

In this section we extend the results of previous sections to the Lie algebra of type $G_2$.

\subsection{Construction of Centralizer of $G_2$}
\label{sec:CentrlzrG2}
Define the root system for $G_2$ where, for convenience, we will use notation $\beta_{i,j}$ for $\alpha_{-i, -j}$: 
$$
 \Delta  =  \{ \alpha_{01}, \alpha_{10}, \alpha_{11}, \alpha_{21}, \alpha_{31}, \alpha_{32},\beta_{32}, \beta_{31 }, \beta_{21}, \beta_{11}, \beta_{10}, \beta_{01}\},  $$                  
where $\alpha_{ij}=i\alpha_{10}+j\alpha_{01}$. 
Fix a Chevalley basis:
$e_{01} = E_{3 1} + E_{6 4}$, 
$ f_{01} = E_{13} + E_{4 6}, $                                                        
$e_{10} = 2 E_{1 7} + E_{2 3} -   E_{45} - E_{7 6},$
$ f_{10} =   E_{32} - E_{5 4} - 2 E_{67} + E_{71}$,
$e_{11}  = - E_{2 1} + 2 E_{3 7} - E_{6 5} + E_{7 4}$,
 $f_{11} = - E_{1 2} + 2 E_{4 7} - E_{5 6} + E_{7 3},$
 $e_{21}  =   E_{14} + 2 E_{2 7} + E_{3 6} + E_{7 5},$
 $f_{21}  =   E_{4 1} + 2 E_{5 7} + E_{6 3} + E_{7 2}$,
$e_{31}  =   E_{15} + E_{2 6}$,
 $f_{31}  =   E_{5 1} + E_{6 2}$,
$e_{32}  = - E_{2 4} + E_{3 5}$,
 $f_{32}  = - E_{4 2} + E_{5 3}$,
$h_{01} = - E_{1 1} +  E_{3 3} - E_{4 4} + E_{6 6}$,
 $h_{10} = 2 E_{1 1} +  E_{2 2} - E_{3 3} + E_{4 4} -  E_{55} - 2E_{6 6}$,
$h_{31}  =   E_{1 1} +  E_{2 2} - E_{5 5} - E_{6 6}$,    
 $h_{21}  =   E_{1 1} + 2E_{2 2} + E_{3 3} - E_{4 4} - 2E_{5 5} -  E_{6 6}$.                                         

Here we are using a standard realization of $G_2$ with matrix units $E_{ij}$. 
All indecomposable lists of roots of  $G_2$ can be obtained by applying Lemma \ref{lem:lemCart_22}. Then we obtain the following description of perfect monomials (we omit the details).

\begin{lemma}\label{lem:lemG2_62}
The following is the set
of all perfect monomials:
\begin{align}  
 &  h_{1} = h_{01},  \,\,\,
 h_{2} = h_{21},  \,\,\,
  c_{1} = f_{01}  e_{01}, \,\,\,
 c_{2} = f_{10}  e_{10}, \,\,\, c_{3} = f_{11}  e_{11}, \,\,\, c_{4} = f_{11}  e_{10}  e_{01}, \,\,\, c_{5} = f_{01}  f_{10}  e_{11}                , \nonumber\\
 & c_{6} = f_{21}  e_{21}, \,\,\,c_{7} = f_{21}  e_{11}  e_{10}, \,\,\, c_{8} = f_{10}  f_{11}  e_{21}, \,\,\, c_{9} = f_{21}  e_{10}^2   e_{01}, \,\,\, c_{10} = f_{01}  f_{10}^2   e_{21}, \,\,\,c_{11} = f_{11}^2   e_{21}  e_{01}   , \nonumber\\
 &c_{12} = f_{01}  f_{21}  e_{11}^2, \,\,\, c_{13} = f_{31}  e_{31}, \,\,\,c_{14} = f_{31}  e_{21}  e_{10}, \,\,\, c_{15} = f_{10}  f_{21}  e_{31}, \,\,\,  c_{16} = f_{31}  e_{11}  e_{10}^2, \,\,\, c_{17} = f_{10}^2   f_{11}  e_{31}    , \nonumber\\
& c_{18} = f_{31}  e_{10}^3   e_{01}, \,\,\,
   c_{19} = f_{01}  f_{10}^3   e_{31}, \,\,\, 
 c_{20} = f_{32}  e_{32}, \,\,\,
 c_{21} = f_{32}  e_{31}  e_{01}, \,\,\,       
   c_{22} = f_{32}  e_{21}  e_{11}, \,\,\,
 c_{23} = f_{01}  f_{31}  e_{32}       , \nonumber\\
& c_{24} = f_{11}  f_{21}  e_{32}, \,\,\,
   c_{25} = f_{32}  e_{21}  e_{10}  e_{01}, \,\,\,
 c_{26} = f_{11}  f_{21}  e_{31}  e_{01}, \,\,\,
 c_{27} = f_{32}  e_{11}^2   e_{10}, \,\,\,           
   c_{28} = f_{01}  f_{31}  e_{21}  e_{11} , \nonumber\\
 &c_{29} = f_{01}  f_{10}  f_{21}  e_{32}, \,\,\, 
c_{30} = f_{10}  f_{11}^2   e_{32}, \,\,\,
  c_{31} = f_{32}  e_{11}  e_{10}^2   e_{01}, \,\,\,
 c_{32} = f_{10}  f_{11}^2   e_{31}  e_{01}, \,\,\,
 c_{33} = f_{01}  f_{31}  e_{11}^2   e_{10} , \nonumber\\    
&   c_{34} = f_{01}  f_{10}^2   f_{11}  e_{32}, \,\,\, 
 c_{35} = f_{32}  e_{10}^3  e_{01}^2, \,\,\,
 c_{36} = f_{01}^2   f_{10}^3   e_{32}, \,\,\,         
   c_{37} = f_{11}^3   e_{32}  e_{01}, \,\,\,          
 c_{38} = f_{01}  f_{32}  e_{11}^3       , \nonumber\\        
& c_{39} = f_{11}^3   e_{31}  e_{01}^2, \,\,\,            
   c_{40} = f_{01}^2   f_{31}  e_{11}^3, \,\,\,
 c_{41} = f_{11}  f_{31}  e_{32}  e_{10}, \,\,\,
 c_{42} = f_{21}^2   e_{32}  e_{10}, \,\,\,              
   c_{43} = f_{10}  f_{32}  e_{31}  e_{11}    , \nonumber\\
& c_{44} = f_{21}^2   e_{31}  e_{11}, \,\,\,
 c_{45} = f_{10}  f_{32}  e_{21}^2, \,\,\,
   c_{46} = f_{11}  f_{31}  e_{21}^2, \,\,\,
 c_{47} = f_{21}^2   e_{31}  e_{10}  e_{01}, \,\,\,
 c_{48} = f_{01}  f_{10}  f_{31}  e_{21}^2     , \nonumber\\
 &  c_{49} = f_{11}  f_{32}  e_{21}^2   e_{01}, \,\,\,
 c_{50} = f_{01}  f_{21}^2   e_{32}  e_{11}, \,\,\,
 c_{51} = f_{21}  f_{31}  e_{32}  e_{10}^2, \,\,\,
   c_{52} = f_{10}^2   f_{32}  e_{31}  e_{21}, \,\,\,
 c_{53} = f_{21}  f_{32}  e_{31}  e_{11}^2     , \nonumber\\
 &c_{54} = f_{11}^2   f_{31}  e_{32}  e_{21}, \,\,\,
   c_{55} = f_{31}^2   e_{32}  e_{10}^3, \,\,\,
 c_{56} = f_{10}^3   f_{32}  e_{31}^2, \,\,\,
 c_{57} = f_{31}  f_{32}  e_{21}^3, \,\,\,              
   c_{58} = f_{2,1}^3   e_{32}  e_{31}            , \nonumber\\
 &c_{59} = f_{21}^3   e_{31}^2   e_{01}, \,\,\,
 c_{60} = f_{01}  f_{31}^2   e_{21}^3, \,\,\,
   c_{61} = f_{32}^2   e_{21}^3   e_{01}, \,\,\,
 c_{62} = f_{32}^2   e_{31}  e_{11}^3, \,\,\,
 c_{63} = f_{11}^3   f_{31}  e_{32}^2, \,\,
   c_{64} = f_{01}  f_{21}^3   e_{32}^2         . \nonumber
 \end{align}
\end{lemma}

Define the order on the set of perfect monomials as follows:
$$
  h_{1} < h_{2} < c_1 < c_2 <  \cdots < c_{63} < c_{64} \,.                          
$$
Consequently, we have
\begin{proposition}
The centralizer subalgebra $U_0(G_2)$  is generated by a finite set of monomials $ \{h_{1}, h_{2}, c_1, \ldots, c_{64} \}$ 
    with a finite number of relations $ \{r_1, \ldots, r_m \}$, 
    where $r_i$ is a polynomial in $h_{1}, h_{2}, c_1, \ldots, c_{64}$ of length  $\leq 4$.    
\end{proposition}

We will use the following quadratic Casimir element of $U_0(\mathfrak{g})$:
\begin{equation*} 
\label{en:refer0219g}
z_1 = 3c_{1} +c_{2} +c_{3} +c_{6}+3c_{13} +3c_{20}   +h_{01}^2 +h_{01}h_{10} +h_{10}^2 +4h_{01} +5h_{10}.
\end{equation*}

Tedious computations show that, using the above relations, the elements $c_{64}, \ldots , c_5$ can be written in terms of   $c_{4}$, $ c_3$, $ c_2$, $c_1$, $h_1$, $h_2$, $z_1$. 

The Lie algebra $G_2$ contains the subalgebra $\hat{\mathfrak{g}}$ of type $A_2$ generated by the following elements
$h_{01}$, $h_{31}$, $e_{01}$, $f_{01}$, $e_{31}$, $f_{31}$, $e_{32}$, $f_{32}$.
We will use the results  from the previous sections just adding ``hat" to all generators and to all formulas.

Consider  a natural embedding $\varphi : \hat{\mathfrak{g}} \rightarrow  \mathfrak{g}$: 
$$ \varphi(\hat h_{01}) = h_{01}, \, \varphi(\hat h_{10}) = h_{31},  \,
 \varphi(\hat e_{01}) = e_{01}, \, \varphi(\hat f_{01}) = f_{01}, \, 
 \varphi(\hat e_{10}) = e_{31}, $$
 $$ \varphi(\hat f_{10}) = f_{31},  \,
 \varphi(\hat e_{11}) = e_{32}, \, \varphi(\hat f_{11}) = f_{32} $$
and extend it to the embedding of the universal enveloping algebras.

The images of the Casimir elements $\hat{z}_1$, $\hat{z}_2$ of $U(\hat{\mathfrak{g}})$ in $U(\mathfrak{g})$ are:
\begin{align*}
Z_{1} & =  \varphi(\hat z_{1})  =
c_{20}+ 
c_{13}+
c_{1}
+ {\frac{1}{3}} ( h_{31}^2
+ 3h_{31}
+ h_{01}^2
+ 3h_{01}
+ h_{31} h_{01})   \label{eq:refer1018a} \\  
Z_{2} & =   \varphi(\hat z_{2}) =
c_{23}+
c_{21}
+ {\frac{1}{3}}   (h_{01}-h_{31})          c_{20}
- {\frac{1}{3}}   (6 +2 h_{01} + h_{31})   c_{13}
+ {\frac{1}{3}}   (h_{01}+2 h_{31})        c_{1}  \nonumber \\
    & + {\frac{1}{27}} (-h_{31}-3+h_{01}) (6+2h_{01}+h_{31}) (h_{01}+2 h_{31}).   \nonumber 
\end{align*}


\subsection{ Torsion free $G_2$-modules}
\label{sec:G2torsionfree}

Let $\Gamma$ be commutative subalgebra of $U_0(G_2)$ generated  by elements
$h_1, h_2, z_1, c_1$. This is our Gelfand-Tsetlin subalgebra for $G_2$.
We will give a
 construction of a $3$-parameter family of $\Gamma$-pointed modules with separating action of $\Gamma$ on basis elements.

Fix $a_1, a_2, a_3\in \mathbb C$ such that $a_3 \notin \mathbb Z$.
Consider the following set of indexed variables: 
\begin{align}
    h_{01}(i,j)    & = a_1 + 2i - j,    \,\,\, h_{10}(i,j)     = \frac{1}{2} ( h_{21}(i,j) - 3 h_{01}(i,j) ) = \frac{1}{2}(a_2 - 3a_1) - 3i + 2j,                      \nonumber\\  
    h_{21}(i,j)    & = a_2 +  j,                 \,\,\,
    h_{11}(i,j)     = \frac{1}{2} ( h_{21}(i,j) + 3 h_{01}(i,j) ) = \frac{1}{2}(a_2 + 3a_1) + 3i -  j,      \nonumber\\
    h_{31}(i,j)    & = \frac{1}{2} ( h_{21}(i,j) - h_{01}(i,j) )   = \frac{1}{2}(a_2 - a_1)  - i + j,        \nonumber\\   
    h_{32}(i,j)    & = \frac{1}{2} ( h_{21}(i,j) + h_{01}(i,j) )   = \frac{1}{2}(a_2 + a_1)  + i,   \,\,\,   s_{jk}          = a_3-j+2k-1,      \nonumber\\                
    S^{+}_{ijk}    & = \frac{1}{2} ( s_{jk}   +  h_{01}(i,j))   = \frac{1}{2}(   a_1 + a_3 + 2 i - 2 j + 2 k -1),  \nonumber\\
    S^{-}_{ik}     & = \frac{1}{2} ( s_{0k}   -  h_{01}(i,0))   = \frac{1}{2}(  -a_1 + a_3 - 2 i       + 2 k -1),  \nonumber\\
    T^{+}_{jk}     & = \frac{1}{2} ( s_{jk} + \frac{1}{3}h_{2,1}(0,j) )  = \frac{1}{6}(   a_1 + 2a_2 + 3a_3 -2j + 6 k -3),  \nonumber\\                           
    T^{-}_{jk}     & = \frac{1}{2} ( s_{jk} - \frac{1}{3}h_{2,1}(0,j) )  = \frac{1}{6}(  -a_1 - 2a_2 + 3a_3 -4j + 6 k -3), \nonumber\\    
    A^{+}_{jk}     & = \frac{ T^{-}_{j-1,k-1} T^{+}_{jk} T^{+}_{j+1,k}   } {9s_{jk}s_{j+1,k}}, \,\,                                                          
    A^{-}_{jk}      = \frac{ T^{-}_{j-1,k-1} T^{-}_{jk} T^{+}_{j+1,k}   } {9s_{jk}s_{j+1,k}},                                                             \nonumber\\
    B^{+}_{jk}     & = \frac{ T^{+}_{j-1,k}   T^{+}_{jk} T^{+}_{j+1,k}   } {27s_{jk}s_{j+1,k}}, \,\,                                                             
    B^{-}_{jk}      = \frac{ T^{-}_{j-1,k-1} T^{-}_{jk} T^{-}_{j+1,k+1} } {27s_{jk}s_{j+1,k}},  \,\, i,j,k \in \mathbb Z.                                                          \nonumber
\end{align}


 Define the action of the Lie algebra $\mathfrak{g}$ on  $V(a_1, a_2, a_3)=\mbox{span}_{\mathbb C} \{{\bf v}_{ijk}\,  |\,  i,j,k \in \mathbb Z\}$ as follows:

$$h_{01}({\bf v}_{ijk})    = h_{01}(i,j)  {\bf v}_{ijk},  \,\,\,
h_{10}({\bf v}_{ijk})    = h_{10}(i,j)  {\bf v}_{ijk},  \,\,\,
h_{11}({\bf v}_{ijk})    = h_{11}(i,j)  {\bf v}_{ijk},  \,\,\,
h_{21}({\bf v}_{ijk})    = h_{21}(i,j)  {\bf v}_{ijk},  $$
$$h_{31}({\bf v}_{ijk})    = h_{31}(i,j)  {\bf v}_{ijk},   \,\,\,
h_{32}({\bf v}_{ijk})    = h_{32}(i,j)  {\bf v}_{ijk}, \,\,\,
e_{01}({\bf v}_{ijk})   = S^{+}_{i,k}    {\bf v}_{i+1,j,k},  \,\,\,
f_{01}({\bf v}_{ijk})   = S^{-}_{ik}      {\bf v}_{i-1,j,k}, $$ 
$$e_{21}({\bf v}_{ijk})   = T^{+}_{j+1,k}    {\bf v}_{i+1,j+2,k+1},  \,\,\,
f_{21}({\bf v}_{ijk})   = T^{-}_{j-1,k-1}  {\bf v}_{i-1,j-2,k-1},    \,\,\,    
e_{10}({\bf v}_{ijk})   =  3   {\bf v}_{i,j+1,k}    +  A^{+}_{jk} S^{-}_{ik}        {\bf v}_{i,j+1,k+1},   $$
$$f_{10}({\bf v}_{ijk})   = -3   {\bf v}_{i,j-1,k-1}  -  A^{-}_{jk} S^{+}_{ijk}      {\bf v}_{i,j-1,k},     \,\,\,
e_{11}({\bf v}_{ijk})   = -3   {\bf v}_{i+1,j+1,k}  +  A^{+}_{jk} S^{+}_{i+1,j+1,k}  {\bf v}_{i+1,j+1,k+1}, $$
$$f_{11}({\bf v}_{ijk})   = -3   {\bf v}_{i-1,j-1,k-1}+  A^{-}_{jk} S^{-}_{i+1,k+1}    {\bf v}_{i-1,j-1,k},  \,\,\,   
e_{31}({\bf v}_{ijk})   =      {\bf v}_{i+1,j+3,k+1} - B^{+}_{jk} S^{-}_{i+1,k+1}    {\bf v}_{i+1,j+3,k+2}, $$  
$$f_{31}({\bf v}_{,jk})   =      {\bf v}_{i-1,j-3,k-2} - B^{-}_{jk} S^{+}_{i+1,j+1,k}  {\bf v}_{i-1,j-3,k-1}, \,\,\,
e_{32}({\bf v}_{,jk})   = -    {\bf v}_{i+2,j+3,k+1} - B^{+}_{jk} S^{+}_{i+1,j+1,k}  {\bf v}_{i+2,j+3,k+2}, $$  

$$f_{32}({\bf v}_{ijk})   =      {\bf v}_{i-2,j-3,k-2} + B^{-}_{jk} S^{-}_{i+1,k+1}    {\bf v}_{i-2,j-3,k-1}, \,\,\,   z_1({\bf{v}}_{ijk})=\frac{14}{3} {\bf{v}}_{ijk}.
 $$                                                                         

\begin{lemma}\label{lem-gamma-G2}
The subalgebra $\Gamma$ has a simple spectrum on $V(a_1, a_2, a_3)$, and hence
separates the basis elements  ${\bf v}_{ijk}$ if and only if $a_3\notin \mathbb Z$. 
\end{lemma}

\begin{proof} It is sufficient to consider vectors from the same weight space. Suppose
$$h_{1}({\bf v}_{ijk}) = (a_1 + 2i - j){\bf v}_{ijk}=\lambda {\bf v}_{ijk}, \,\,
  h_{2}({\bf v}_{ijk}) = (a_2 +  j){\bf v}_{ijk}=\mu {\bf v}_{ijk},$$
for some $\lambda$ and $\mu$. Then $j=\mu-a_2$ and $i=\frac{1}{2}(\lambda+\mu -a_1-a_2)$. Hence, basis elements of this weight subspace differ by the third index. 
Consider 
${\bf v}_{ijk_1}$ and ${\bf v}_{ijk_2}$, for arbitrary $i,j,k_1,k_2$. We have
$$e_{01}f_{01}({\bf v}_{ijk})=S^{-}_{ik} S^{+}_{i-1,j,k}({\bf v}_{ijk})=\frac{1}{4}(  -a_1 + a_3 - 2 i       + 2 k -1 )(   a_1 + a_3 + 2 i - 2 j + 2 k -3 ){\bf v}_{ijk}.$$
From here we obtain that $c_1$ separates  
the basis elements ${\bf v}_{ijk_1}$ and ${\bf v}_{ijk_2}$ implying  the statement.
\end{proof}


\begin{theorem}\label{prop:G2module} 
For any complex $a_1, a_2, a_3$ such that $a_3 \notin \mathbb Z$, $V(a_1, a_2, a_3)$ has a structure of a $G_2$-module.
\end{theorem}

\begin{proof} The statement follows by checking that the defining relations of the Lie algebra $G_2$ are satisfied on $V(a_1, a_2, a_3)$. We omit further the details.
\end{proof}

\medskip

\begin{theorem}\label{prop:Prop21-G2}  Let $a_1, a_2, a_3\in \mathbb C$ and $a_3 \notin \mathbb Z$. Then
\begin{itemize}
\item[1.] $V(a_1, a_2, a_3)$ is a  torsion free simple $\Gamma$-pointed  $G_2$-module 
 if and only if $S^{+}_{ijk}S^{-}_{ik}T^{+}_{jk}T^{-}_{jk}   \neq 0$, for all $i,j,k  \in \mathbb Z$. 

\item[2.] If $V'$ is a simple torsion free $\Gamma$-pointed  $G_2$-module with a basis parametrized by the lattice $\mathbb Z^3$ and with separating action of $\Gamma$ on basis elements,
then it is
     isomorphic to $V(a_1, a_2, a_3)$ 
    for  suitable parameters $a_1$, $a_2$, $a_3$ such that
    $0 \le Re \,a_1 < 1$, $0 \le Re \,a_2 < 3$, $0 < Re \,a_3 < 2$ and $a_3\neq 1$.

\end{itemize}
\end{theorem}

\begin{proof} It follows immediately from the formulas of the action of $G_2$ that $V(a_1, a_2, a_3)$ is a  torsion free  module 
 if and only if $S^{+}_{ijk}\neq 0$, $S^{-}_{ik} \neq 0$, $T^{+}_{jk}\neq 0$ and $T^{-}_{jk}\neq 0$,  for all $i,j,k  \in \mathbb Z$.
Note that $\Gamma$ has a simple spectrum on $V(a_1, a_2, a_3)$ and hence, the action of $\Gamma$ separates the basis elements by Lemma \ref{lem-gamma-G2}. In particular, $V(a_1, a_2, a_3)$ is $\Gamma$-pointed. Using this fact, it follows that  
conditions  $S^{+}_{ijk}S^{-}_{ik}T^{+}_{jk}T^{-}_{jk}   \neq 0$, for all $i,j,k  \in \mathbb Z$, are necessary and sufficient to guarantee that any element of $V(a_1, a_2, a_3)$ generates the whole module, and hence the simplicity of the module. 

 Suppose that $V'$ is a simple torsion free $\Gamma$-pointed  $G_2$-module with a basis $\{v'_{ijk}, \, i,j,k \in \mathbb Z\}$ such that $\Gamma$ acts by different characters on the basis elements $v'_{ijk}$. Then the same argument as in the proof of Theorem \ref{prop:Prop21-C2} shows that $V'\simeq  V(a_1, a_2, a_3)$ for some choice of complex parameters $a_1, a_2, a_3$ with $a_3\notin \mathbb Z$. Since module $V'$ is torsion free, then we can move to a different weight space and, hence the parameters can be chose in such a way that their real parts satisfy the inequalities $0 \le Re \,a_1 < 1$, $0 \le Re \,a_2 < 3$, $0 < Re \,a_3 < 2$. This completes the proof.
\end{proof}

Let us define a homomorphism  $\theta : A_2 \rightarrow  G_2$ of Lie algebras as follows: 
$
\theta( \hat e_{01} ) =   e_{01}$,    $\theta( \hat f_{01} ) =   f_{01}$,
$\theta( \hat e_{10} ) =   e_{31}$,           $ \theta( \hat f_{10} ) =   f_{31}$,  
$\theta( \hat e_{11} ) =   e_{32}$,            $ \theta( \hat f_{11} ) =   f_{32}$, 
$\theta( \hat h_{01} ) =   h_{01}$,           $ \theta( \hat h_{10} ) =   h_{31}$,  
and denote by $\hat{\mathfrak{g}}$ its image. Consider the restriction $\hat{V}(a_1,a_2,a_3)$ of 
 $V(a_1,a_2, a_3)$ on $\hat{\mathfrak{g}}$. 
The images $Z_1$ (respectively, $Z_2$)
of  the Casimir elements of $A_2$ act on $\hat{V}(a_1,a_2,a_3)$ by $ \frac{-8}{9}$ (respectively,  $\frac{8}{9}$).
Then, we have the following decomposition of $\hat{V}(a_1,a_2,a_3)$.

\medskip
\begin{theorem}\label{prop:Prop26} 
Suppose that  $V(a_1,a_2,a_3)$ is torsion free. 
 The $A_2$-module $\hat{V}(a_1,a_2,a_3)$ decomposes into a direct sum of three torsion free submodules
$\hat{V}(a_1,a_2,a_3) = \hat{V}{(1)} \oplus \hat{V}{(2)} \oplus \hat{V}{(3)},$ 
where 
$$\hat{V}{(m)}= \mbox{span}_{\mathbb C}\{{\mathbf v}_{i,3j+m,k}\,  | \, i,j,k \in \mathbb Z\}\simeq W=V( a^{(m)}_1, a^{(m)}_2, a_3, \frac{-8}{9}, \frac{8}{9})\, ,$$ 
$a^{(m)}_1=a_1-m$, $a^{(m)}_2=\frac{1}{2}(a_2-a_1)+m$, with $m=0,1,2$.
\end{theorem}

\begin{proof}
Consider the subspace $\hat{V}{(m)}= \mbox{span}\{{\mathbf v}_{i,3j+m,k} \, | \, i,j,k \in \mathbb Z\}$ of $\hat{V}(a_1,a_2,a_3)$, for $m=0,1,2$. Clearly, $\hat{V}{(m)}$ is a  $\hat{\mathfrak{g}}$-submodule of $\hat{V}(a_1,a_2,a_3)$.
The $\hat{\mathfrak{g}}$-module
$W=\mbox{span}_{\mathbb C}\{w_{ijk}\, | \, i,j,k \in \mathbb Z\}$ is defined by   \eqref{eq:REF109}-\eqref{eq:REF_A2418},
and we have a 
homomorphism of $\hat{\mathfrak{g}}$-modules 
$\psi_m: W\rightarrow \hat{V}{(m)},$
such that
$\psi_m:w_{ijk} \mapsto v_{i+j,3j+m,k+j}$,
 and,
for  any $x \in \hat{\mathfrak{g}}$,  
$\psi_m ( x ( w_{ijk} )) = 
 \theta(x) ( \psi_m ( w_{ijk} )) = 
 \theta(x) (  v_{i+j,3j+m,k+j})$ holds, with 
$i,j,k\in \mathbb Z$ and $m=0,1,2$.
Since $\psi_m$ is a linear isomorphism, the proof is completed.
\end{proof}

\section*{Comment}
\noindent The detailed computations are avilable at \href{https://icm.sustech.edu.cn/people/VyacheslavFutorny}{https://icm.sustech.edu.cn/people/VyacheslavFutorny}.

\section*{Acknowledgments}	
  \noindent This work was supported by Kuwait University, Research Grant 02/22.  The authors are very grateful to the referees for many useful suggestions.

\end{document}